\documentclass[12 pt,reqno]{amsart}
\usepackage{ upgreek }
\usepackage{latexsym, mathrsfs, amsmath, bm, amssymb, longtable, booktabs,amscd,microtype,booktabs,cases}
\usepackage[square,numbers]{natbib}
\usepackage{graphicx}
\usepackage[dvipsnames]{xcolor}
\usepackage{caption}
\usepackage{dsfont}
\usepackage[all]{xy}
\usepackage{enumerate}
\usepackage{amsthm}
\usepackage{multirow}
\usepackage{tikz}
\usetikzlibrary{matrix,arrows,decorations.pathmorphing}
\usepackage{collectbox}
\usepackage{enumerate}
\usepackage[colorinlistoftodos,prependcaption, textsize=tiny]{todonotes}
\reversemarginpar
\usepackage{amsmath}
\usepackage{enumitem}
\usepackage{hyperref}
\let\oldtocsection=\tocsection

\let\oldtocsubsection=\tocsubsection

\let\oldtocsubsubsection=\tocsubsubsection

\renewcommand{\tocsection}[2]{\hspace{0em}{\vspace{0.5em}}\oldtocsection{#1}{#2}}
\renewcommand{\tocsubsection}[2]{\hspace{1em}{\vspace{0.5em}}\oldtocsubsection{#1}{#2}}
\renewcommand{\tocsubsubsection}[2]{\hspace{2em}\oldtocsubsubsection{#1}{#2}}
\hypersetup{
	colorlinks,
	citecolor=red,
	filecolor=black,
	linkcolor=blue,
	urlcolor=black
}
\numberwithin{equation}{section} 
\usepackage{mathtools}

\newcommand{\Z}{\mathbb{Z}}

\newcommand{\C}{\mathbb{C}}

\makeatother

\newtheorem{thm}{Theorem}[section]

\newtheorem{cor}[thm]{Corollary}

\newtheorem{prop}[thm]{Proposition}
\newtheorem{lemma}[thm]{Lemma}
\theoremstyle{definition}
\newtheorem{definition}[thm]{Definition}
\newtheorem{remark}[thm]{Remark}

\theoremstyle{definition}

\theoremstyle{remark}

\theoremstyle{remark}

\makeatletter
\def\imod#1{\allowbreak\mkern10mu({\operator@font mod}\,\,#1)}
\makeatother

\begin{document}

\title{Weyl modules for twisted toroidal Lie algebras}
\author{Ritesh Kumar Pandey, Sachin S. Sharma}
\date{\today}

\begin{abstract}In this paper, we extend the notion of Weyl modules for twisted toroidal Lie algebra $\mathcal{T}(\mu)$. We prove that the level one global Weyl modules of $\mathcal{T}(\mu)$ are isomorphic to the tensor product of the level one representation of twisted affine Lie algebras and certain lattice vertex algebras. As a byproduct, we calculate the graded character of the level one local Weyl modules of $\mathcal{T}(\mu)$.
\end{abstract}
\maketitle

{\bf{Notations:}} 
\begin{itemize}
\item The sets of complex numbers, integers, non-negative integers, and positive integers are denoted by $\C$, $\Z$, $\Z_{\geq 0}$, and $\Z_{>0}$, respectively.
\item Let $\C^n\coloneqq \{(x_1, x_2,\dots,x_n) : x_i \in \C, 1 \leq i \leq n\}$, and the sets $\Z^n$, $\Z_{\geq 0}^{n}$, and $\Z_{>0}^{n}$ are defined similarly.
\item For $n\geq 2$, we denote the elements $(m_1,\dots,m_n)\in\Z^n$ and $(m_2,\dots,m_n)\in\Z^{n-1}$ by $\bf{m}$ and $\underline{m}$, respectively.
\item Let $A_n\coloneqq\C[t_1^{\pm1},t_2^{\pm1},\dots,t_n^{\pm1}]$ be the set of Laurent polynomial ring in $n$ variables.
\item For ${\bf{m}}=(m_1,\dots,m_n)\in \Z^n$ and ${\underline{m}}=(m_2,\dots,m_n)\in \Z^{n-1}$ the notation $t^{\bf{m}}$ and $t^{\underline{m}}$ represents $t_1^{m_1}\cdots t_n^{m_n}$ and $t_2^{m_2} \cdots t_n^{m_n}$, respectively, in $A_n$.
\item Let $U(\mathfrak{g})$ denotes the universal enveloping algebra of $\mathfrak{g}$ and $L_n(\mathfrak{g})$ denotes the Lie algebra $\mathfrak{g}\otimes A_n$ with commutator given by $[x\otimes t^{\bf{m}}, y\otimes t^{\bf{n}}] = [x,y]\otimes t^{\bf{m+n}}$.
\end{itemize}

\section{Introduction}An attempt to make inroads in the representation theory of quantum affine algebras in both classical and quantum settings led to the birth of the notion of Weyl modules. It was originally defined for loop algebras $L(\mathfrak{g})$ by Chari and Pressley \cite{VCAP}. The global Weyl modules are indexed by the dominant weights of $\mathfrak{g}$ and are the maximal integrable modules of $L(\mathfrak{g})$. The local Weyl modules are a maximal finite dimensional quotient of global Weyl modules. The local Weyl modules can be identified with graded limits of Drinfeld polynomials of irreducible representation modulo some constraints \cite{DVG}. In the last two decades, the notion of Weyl modules has been extensively studied and their connections with Demazure modules or Kirillov-Reshetikhin modules have been investigated \cite{LLS, VGT, VCTL, RFS, FBLS, FL, FMS, KR, SSS, KN}. For the current Lie algebra $\mathfrak{sl}_{n}[t]$, it was proved that the graded character of the local Weyl module can be given by specializing a Macdonald polynomial at $t= 0$ \cite{BBAC}. This result was a crucial ingredient in the proof of BGG reciprocity for $\mathfrak{sl}_{n}[t]$. Because of these interesting aspects, it is natural to extend this notion to different classes of Lie algebras. The Weyl modules were defined and studied for twisted loop algebras in \cite{VCAP} and for loop Kac-Moody algebras in \cite{RFS}. The notion of Weyl modules was studied for the quotient of toroidal Lie algebra in two variables in \cite{VCTL}, where the quotient excludes a large part of the center. Recently, Kodera removed this constraint and defined Weyl modules for toroidal Lie algebra in two variables \cite{KR}. Further, using the results in \cite{RAOPL}, he obtained a realization of level one global Weyl modules and consequently calculated the graded character of the level one local Weyl modules. In \cite{SSS}, authors generalized the results of Kodera to the toroidal Lie algebras in arbitrary $n$ variables.

On the other hand vertex algebras and their representation theory play an important role in constructing representations of infinite dimensional Lie algebras, especially toroidal Lie algebras \cite{REB, MR1, FEBZ, FILM, K1, KRR, HL1, MRY}. In \cite{MRY}, using vertex operators, a large class of integrable, indecomposable representations were constructed for untwisted toroidal Lie algebras of two variables. These results were generalized for $n$ variables in \cite{MR1}. Using the vertex operator realizations, Billig classified the irreducible modules of the bounded category for full toroidal Lie algebras \cite{YULYB}. The twisted vertex operators are useful in the construction of principal realizations of affine Kac-Moody algebras and moonshine vertex operator algebras \cite{TG, KKLW, LJWR, LW1}. The twisted representation of vertex algebras, which are defined using twisted vertex operators, played a significant role in the study of representations of twisted toroidal Lie algebras and extended affine Lie algebras \cite{BKIRK, YULYBL}. 

Let $\mathfrak{g}$ be a simple finite dimensional Lie algebra of type $A, D$ or $E$. Let $\mu$ be a Dynkin diagram automorphism of $\mathfrak{g}$ of order $r$. Let $\mathcal{T}(\mu)$ be the twisted toroidal Lie algebra (see Subsection \ref{sec5}) and let $\mathcal{T}_{\mathrm{aff}}(\mu)$ be the twisted affine Kac-Moody algebra contained in $\mathcal{T}(\mu)$. For every dominant integral weight $\Lambda$ of $\mathcal{T}_{\mathrm{aff}}(\mu)$, we define the global Weyl module $W_{\mathrm{glob}}^{\mu}(\Lambda)$ of $\mathcal{T}(\mu)$. For every $W_{\mathrm{glob}}^{\mu}(\Lambda)$, one associates a commutative associative algebra $A^{\mu}(\Lambda)$, which defines a right action on $W_{\mathrm{glob}}^{\mu}(\Lambda)$. The local Weyl module $W_{\mathrm{loc}}^{\mu}(\Lambda)$ is defined as the tensor product of $W_{\mathrm{glob}}^{\mu}(\Lambda)$ with one dimensional representation of $A^{\mu}(\Lambda)$. In this paper, for the dominant integral weight of level one, we prove that the corresponding global Weyl module is isomorphic to the tensor product of the basic representation of $\mathcal{T}_{\mathrm{aff}}(\mu)$ and certain lattice vertex algebra of $\mathcal{T}(\mu)$. Building on this realization, we derive the graded character formula for $W_{\mathrm{loc}}^{\mu}(\Lambda)$.
Let $L(\Lambda_0)$ be the basic representation of twisted affine Lie algebra $\mathfrak{g}(\mu)$ and $V_{\Gamma_1}$ be the lattice vertex algebra associated with certain lattice $\Gamma_1$. The following are our main results (see Section \ref{secv} for notation details):

\begin{thm}The space 
$L(\Lambda_0) \otimes V_{\Gamma_1}$ acquires a structure of $\mathcal{T}(\mu)$-module. Furthermore, as a $\mathcal{T}(\mu)$-module, $L(\Lambda_0) \otimes V_{\Gamma_1} \cong W^{\mu}_{\mathrm{glob}}(\Lambda_{0})$.
\end{thm}

\begin{thm}
Let $\mathcal{T}(\mu)$ be a twisted toroidal Lie algebra whose underlying finite dimensional simple Lie algebra is of type $A_{2 \ell +1}$, or $D_{\ell}$. Let $\mathcal{T}_{\mathrm{aff}}(\mu)$ be twisted affine Kac-Moody subalgebra of $\mathcal{T}(\mu)$ and let $L(\Lambda_0)$ be its basic representation. Then we have
$$\mathrm{ch}_{q_1} W^{\mu}_{\mathrm{loc}}(\Lambda_0, \underline{a}) =
\mathrm{ch}_{q_1} L(\Lambda_0)\left( \prod_{s>0}{\frac{1}{1-q_1^{s}}}\right)^{n-1},$$
$$\mathrm{ch}_{q_1, q_2, \ldots, q_n}
W^{\mu}_{\mathrm{loc}}{(\Lambda_0, \underline{a})} =
\mathrm{ch}_{q_1} L(\Lambda_0) \left( \prod_{s>0, i = 2}^{n}{\frac{1}{1-q_1^{s} q_i}}\right).$$
\end{thm}

\begin{remark}
We would like to emphasize that our results do not include the twisted toroidal Lie algebras whose underlying finite dimensional simple Lie algebras are of types $A_{2 \ell}$ and $E_{6}$. For the $A_{2 \ell}$ case, the Lemma \ref{RP1}, which plays a crucial role in this paper, does not hold, and for
$E_{6}$, we do not have realization of corresponding twisted toroidal Lie algebra in terms of MRY presentation, which plays a key role in the definition of the global Weyl modules.
\end{remark}
\subsection{{\bf{Organization of the paper}:}} We begin with the basics of twisted affine Lie algebras, toroidal Lie algebras, and twisted toroidal Lie algebras in Section \ref{secp}. In Section \ref{secpr}, we use the results of \cite{JMMishra} to obtain the realization of $\mathcal{T}{(\mu)}$ in terms of generators and relations. Using the realization, we first define global Weyl modules for $\mathcal{T(\mu)}$ in Section \ref{secgl}. Using Garland identities, we prove that for a dominant integral weight $\Lambda$, the commutative associative algebra $A^{\mu}(\Lambda)$ is associated with $W_{\mathrm{glob}}^{\mu}(\Lambda)$ is isomorphic to the ring of invariants of the subgroup of the permutation group (Subsection \ref{thm1}) and prove the result related to the weight spaces of global and local Weyl modules (Proposition \ref{prop2}). Section \ref{secl1} is devoted to the dominant integral weight of level one, where we obtain an upper bound for the character of the level one local Weyl module (Proposition \ref{prop1}). We start Section \ref{secv} with the basics of vertex algebras. Using result of \cite{BKIRK}, we obtain the realization of level one global Weyl module in terms of tensor product of modules over twisted affine Kac-Moody algebra and lattice vertex algebra (Theorem \ref{gthm}). Finally, as a corollary, we derive the character formula of the level one local Weyl module.

\section{Preliminaries} {\label{secp}} 
Let $\mathfrak{g}$ be a finite dimensional simple Lie algebra of type $A_{2\ell}\,\,(\ell\geq 1),$ $A_{2\ell-1}\,\,(\ell\geq 2),$ $D_{\ell+1}\,\,(\ell\geq 3)$  over $\C$ and $\mathfrak{h}$ be its Cartan
subalgebra. Let $I$ denote the index set for the simple roots of $\mathfrak{g}$. We fix a set of simple roots ${\alpha_i'}$, simple coroots ${{\alpha_i'}}^\vee$, and fundamental weights $\Lambda_i'$ for $i\in I$. We denote by $R$ the set of roots, $R(+)$ (respectively  $R(-)$) the set of positive (respectively negative) roots of $\mathfrak{g}$ with respect to $\mathfrak{h}$. Let $(\cdot|\cdot)$ be a nondegenerate invariant symmetric bilinear form on $\mathfrak{g}$ defined by $(x|y)=\text{tr}(xy),\,\, \text{tr}(xy),\,\, \text{and}\,\, \frac{1}{2}\text{tr}(xy)$ for $\mathfrak{g}=$ $A_{2\ell},$ $A_{2\ell-1},$ $\text{and}\,\, D_{\ell+1},$ respectively.  For any ${\alpha'} \in R^+$, let $\mathfrak{g}_{\pm {\alpha'}}$ be the corresponding root spaces and fix non-zero elements $e'_{\pm {\alpha'}}\in \mathfrak{g}_{\pm {\alpha'}}, h'_{\alpha'} \in \mathfrak{h}$ such that $[e'_{\alpha'},e'_{-{\alpha'}}]=h'_{\alpha'}, [h'_{\alpha'},e'_{\pm {\alpha'}}]=\pm 2e'_{\pm {\alpha'}}$. Then the subalgebra generated by $\{e'_{\alpha'},f'_{\alpha'},h'_{\alpha'}\}$ is isomorphic to $\mathfrak{sl}_2$. Set $e'_{{\alpha_i'}}=e'_i, e'_{-{\alpha_i'}}=f'_i$, and $h'_{{\alpha_i'}}=h'_i$. The set of elements $e'_i$, $f'_i$, and $h'_i$, $i\in I$ generates the Lie algebra $\mathfrak{g}$ known as the Chevalley generators. We have the standard triangular decomposition of $\mathfrak{g}$, $\mathfrak{g}= n^-\oplus \mathfrak{h} \oplus n^+$, where $n^-(\text{respectively }\,\, n^+)$ is the subalgebra of $\mathfrak{g}$ generated by $f'_i$ (respectively  $e'_i$), $i\in I$.

Let $\mu$ be a diagram automorphism of $\mathfrak{g}$ of order $r \text{ (which is }\text{$2$ or $3$})$ and let $\xi$ be an $r^{th}$ root of unity (say $\xi=$exp$(2\pi i/r)$). Since $\mu$ is diagonalizable, $\mathfrak{g}$ has the following $\Z/r\Z$ grading:
$$
\mathfrak{g}=\bigoplus_{s={0}}^{ {r-1}}\mathfrak{g}_s,
$$
where $\mathfrak{g}_s=\{x\in\mathfrak{g}: \mu(x)=\xi^sx\}$. It is well known that the subalgebra $\mathfrak{g}_0$ is a simple Lie algebra and $\mathfrak{g}_j$, for $j= {1},\dots, {r-1}$, are irreducible $\mathfrak{g}_0$-modules. The following table describes the various possibilities for $\mathfrak{g}$ and $\mathfrak{g}_0$:

 \begin{center}$
    \begin{tabular} { | c | c | c | c | }
    \hline
     $r$ & $\mathfrak{g}$ &$\mathfrak{g}_0$ \\
    \hline
     $2$ & $A_{2\ell}, \ell\geq 1$ & $B_\ell$ \\
     $2$ & $A_{2\ell-1}, \ell\geq 2$ & $C_\ell$\\
    $2$ & $D_{\ell+1},\ell\geq 3$ & $B_\ell$\\
    $3$ & $D_4$& $G_2$\\
    \hline
    \end{tabular}$
    \end{center}

     Let $I_0=\{1,2,\dots,\,$rank$\{\mathfrak{g}_0\}\}$. The Chevalley generators $\{e^{(0)}_i,f^{(0)}_i,h^{(0)}_i| i\in I_0\}$ of $\mathfrak{g}_0$ are given by:
     $$ e^{(0)}_i= e'_i,\, f^{(0)}_i=f'_i,\, h^{(0)}_i= h'_i, \, \text{if} \,\mu({\alpha_i'})= {\alpha_i'};$$
     $$e^{(0)}_i=\sum_{j=1}^{r-1}\mu^j(e'_{i}),\, f^{(0)}_i=\sum_{j=1}^{r-1}\mu^j(f'_{i}),\, h^{(0)}_i=\sum_{j=1}^{r-1}\mu^j(h'_{i}), \, \text{if} \,\mu({\alpha_i'})\neq {\alpha_i'}, \, \text{for} \, \mathfrak{g}=A_{2\ell},\, i\neq \ell;$$
     $$e^{(0)}_\ell=\sqrt{2}(e'_\ell+e'_{\ell+1}),\, f^{(0)}_\ell=\sqrt{2}(f'_\ell+f'_{\ell+1}),\, h^{(0)}_\ell=2(e'_\ell+e'_{\ell+1})\, \text{for} \, \mathfrak{g}= A_{2\ell}.$$
     We have the standard triangular decomposition of $\mathfrak{g}_0$, $\mathfrak{g}_0= n_0^-\oplus \mathfrak{h}_0 \oplus n_0^+$. Where $n_0^-\,(\text{respectively }\,\,n_0^+)$ is the subalgebra of $\mathfrak{g}$ generated by $f^{(0)}_i$ (respectively  $e^{(0)}_i$), $i\in I_0$, and $\mathfrak{h}_0$ is the Cartan subalgebra of $\mathfrak{g}_0$ spanned by  $h^{(0)}_i,\, i\in I_0$.
     The simple roots $\alpha_i\in\mathfrak{h}_0^*,\, i\in I_0$ of $\mathfrak{g}_0$ are given by:
     $${\alpha}_i = \frac{1}{r}\sum_{j=1}^{r-1}\mu^j({\alpha_i'}).
     $$ Let $\alpha_i^\vee\in\mathfrak{h}_0,\, i\in I_0$ be the simple coroots of $\mathfrak{g}_0$. Note that the restriction of the nondegenerate invariant symmetric bilinear form $(\cdot|\cdot)$ on both $\mathfrak{g}_0$ and $\mathfrak{h}_0$  remains nondegenerate.                            
     
     Let $A=(a_{i,j})_{i,j\in I_0}$ be the Cartan matrix of $\mathfrak{g}_0$. We introduce the element $\theta_0\in \mathfrak{h}$ as follows:
     $$\theta_0=
       \begin{cases}
    {\alpha_1'}+\cdots+{\alpha_{2\ell}'}, & \text{for $\mathfrak{g}=A_{2\ell}$}.\\[.3em]
    {\alpha_1'}+\cdots+{\alpha_{2\ell-2}'}, & \text{for $\mathfrak{g}=A_{2\ell-1}$}.\\[.3em]
    {\alpha_1'}+\cdots+{\alpha_{\ell}'}, & \text{for $\mathfrak{g}=D_{\ell+1}$}.\\[.3em]
    {\alpha_1'}+{\alpha_2'}+{\alpha_3'}, & \text{for $\mathfrak{g}=D_{4}$}.
  \end{cases}$$

  Let the $\mathfrak{sl}_2$-triplet associated with $\theta_0$ be denoted by $e'_{\theta_0}, f'_{\theta_0}, h'_{\theta_0}$ with the bracket $[h'_{\theta_0},e'_{\theta_0}]=2e'_{\theta_0}$,  $[h'_{\theta_0},f'_{\theta_0}]=-2f'_{\theta_0}$ and $[e'_{\theta_0},f'_{\theta_0}]=h'_{\theta_0}.$

 Let $$\theta^0=\begin{cases}
 \theta_0, & \text{for $\mathfrak{g}=A_{2\ell}$}
 \\[.3em] \frac{1}{r}\sum_{j=1}^{r-1}\mu^j(\theta_0), &\text{for $\mathfrak{g}=A_{2\ell-1}, D_{\ell+1},D_{4}.$}
\end{cases}$$
 \begin{remark}{} For any $0\leq j\leq r-1$, we denote an element $x\in \mathfrak{g}_j$ by $x^{(j)}$.  Note that for any $k\in \Z$, $x^{(k)}=x^{(j)}$ where $k=rs+j$ for some $s\in \Z$ and $0\leq j\leq r-1$.
\end{remark}

 Let $\theta_s$ be the highest short root of $\mathfrak{g}_0$; in fact, $\theta_s=\frac{1}{2}\theta^0\, \text{for} \, \mathfrak{g}=A_{2\ell}$, $\theta_s=\theta^0\,\text{for} \, \mathfrak{g}=A_{2\ell-1},\, D_{\ell+1},\, D_4.$ For $\mathfrak{g}=A_{2\ell-1}, D_{\ell+1},\,\text{and}\, D_4$, let
$$e^{(j)}_{\theta_s}=:\sum_{i=1}^{r-1}\mu^i(\xi^{(r-i)j}e'_{\theta_0})\in {(\mathfrak{g}_j)}_{\theta_s},\,\,f^{(j)}_{\theta_s}=:\sum_{i=1}^{r-1}\mu^i(\xi^{(r-i)j}f'_{\theta_0})\in {(\mathfrak{g}_j)}_{-\theta_s},$$
and
$$h^{(j)}_{\theta_s}=:\sum_{i=1}^{r-1}\mu^i(\xi^{(r-i)j}h'_{\theta_0})\in {(\mathfrak{g}_j)}_{0},$$
where $0\leq j\leq r-1$. By direct calculation, we have the following bracket operations for $e^{(j)}_{\theta_s},f^{(j)}_{\theta_s},$ and $h^{(j)}_{\theta_s}$:
$$
\begin{aligned}
[e^{(i)}_{\theta_s},f^{(j)}_{\theta_s}]=
 h^{(i+j)}_{\theta_s},
[h^{(i)}_{\theta_s},e^{(j)}_{\theta_s}]=
2e^{(i+j)}_{\theta_s} ,
 [h^{(i)}_{\theta_s},f^{(j)}_{\theta_s}]=
-2f^{(i+j)}_{\theta_s},
\end{aligned}
$$
and
$$[e^{(i)}_{\theta_s},e^{(j)}_{\theta_s}]=[f^{(i)}_{\theta_s},f^{(j)}_{\theta_s}]=[h^{(i)}_{\theta_s},h^{(j)}_{\theta_s}]=0,\,\,\text{for all}\,\,1\leq i,j\leq r-1.$$
By direct calculation and using [\cite{KAC}, Lemma 8.1. a)] we have $$
\begin{aligned}
(e^{(r-i)}_{\theta_s}|f_{\theta_s}^{(j)})=&\begin{cases}
0,& \text{if}\,\, 0\leq i\neq j\leq r-1,\\[.3em]
r,& \text{if}\,\, 0\leq i= j\leq r-1.
\end{cases}
\end{aligned}
$$
      
Note that $f^{(1)}_{\theta_s}$ (respectively  $e^{(r-1)}_{\theta_s}$) is the highest (respectively  lowest) weight vector of the $\mathfrak{g}_0$-module $\mathfrak{g}_1$ (respectively  $\mathfrak{g}_{r-1}$) with weight $\theta^0$ (respectively  $-\theta^0$) as like in [\cite{KAC}, Proposition 8.3 e)].

\subsection{Untwisted affine Lie algebras} Let $\mathfrak{g}$ be a finite dimensional simple Lie algebra. Consider the loop algebra $L(\mathfrak{g})=\mathfrak{g}\otimes\C[t^{\pm1}]$. This is an infinite dimensional Lie algebra with the bracket defined by:

$$[x\otimes t^m,y\otimes t^n]=[x,y]\otimes t^{m+n},$$
where $x,y\in\mathfrak{g}$ and $m,n\in\Z.$ Let $\tilde{L}(\mathfrak{g})$ denote the one-dimensional central extension of loop algebra $L(\mathfrak{g})$ given by $\tilde{L}(\mathfrak{g})=\mathfrak{g}\otimes\C[t^{\pm1}]\oplus\C K$. Further, by adjoining a derivation $d$, we extend $\tilde{L}(\mathfrak{g})$ as an algebra $\hat{\mathfrak{g}}=\mathfrak{g}\otimes\C[t^{\pm1}]\oplus\C K\oplus\C d$. Then the algebra $\hat{\mathfrak{g}}$ is called an untwisted affine Lie algebra with the following Lie bracket:
$$[x\otimes t^m,y\otimes t^n]=[x,y]\otimes t^{m+n}+m\delta_{m+n,0}(x|y)K,$$
$$[\hat{\mathfrak{g}},K]=0,\,\,[d,x\otimes t^m]=mx\otimes t^m,$$
where $x,y\in\mathfrak{g}$ and $m,n\in\Z.$ The normlized invariant bilinear form on $\hat{\mathfrak{g}}$ is given by:
$$(x\otimes P_1|y\otimes P_2)=\mathrm{Res}(t^{-1}P_1P_2)(x|y),\,\,(\C K+\C d|x\otimes P_1)=0,$$
$$(K|K)=(d|d)=0,\,\,(K|d)=1,$$
where $x,y\in\mathfrak{g},$ $P_1,P_2\in\C[t^{\pm 1}]$ and $\mathrm{Res}(t^{-1}P_1P_2)$ is the residue of a Laurent polynomial $t^{-1}P_1P_2$. The bilinear form on $\hat{\mathfrak{g}}$ is nondegenerate and symmetric.

\subsection{Twisted affine Lie algebras} \label{sec3} Consider the untwisted affine Lie algebra
$$\hat{\mathfrak{g}}=\mathfrak{g}\otimes\C[t^{\pm 1}]\oplus \C K\oplus \C d.$$
Let extend the diagram automorphism $\mu$ to $\hat{\mathfrak{g}}$ by 
$$\tilde{\mu}(x\otimes t^j)=\xi^{-j}\mu(x)\otimes t^j,\,\,\tilde{\mu}(K)=K,\,\,\tilde{\mu}(d)=d.$$
The twisted affine Lie algebra $\hat{\mathfrak{g}}({\mu})$ is defined as the set of fixed points of $\tilde{\mu}$ on $\hat{\mathfrak{g}}$, which is a subalgebra of $\hat{\mathfrak{g}}$. In fact,
$$\hat{\mathfrak{g}}(\mu)=\bigoplus _{{j}={0}}^{{r-1}}\mathfrak{g}_j\otimes t^j\C[t^{\pm r}] \oplus \C K\oplus\C d.$$
The normlized invariant bilinear form on $\hat{\mathfrak{g}}(\mu)$ is given by:
$$(x\otimes P_1|y\otimes P_2)=\frac{1}{r}\mathrm{Res}(t^{-1}P_1P_2)(x|y),\,\,(\C K+\C d|x\otimes P_1)=0,$$
$$(K|K)=(d|d)=0,\,\,(K|d)=1,$$
where $x\in\mathfrak{g}_i,y\in\mathfrak{g}_j,$ $P_1\in t^i\C[t^{\pm r}]$, $P_2\in t^j\C[t^{\pm r}]$ for $0\leq i,j\leq r-1$, and $\mathrm{Res}(t^{-1}P_1P_2)$ is as defined above. {The bilinear form on $\hat{\mathfrak{g}}(\mu)$ is nondegenerate and symmetric}. The subalgebra $\hat{\mathfrak{g}}'(\mu)$ of $\hat{\mathfrak{g}}(\mu)$ is given by
$$\hat{\mathfrak{g}}'(\mu)=\bigoplus _{{j}={0}}^{{r-1}}\mathfrak{g}_j\otimes t^j\C[t^{\pm r}] \oplus \C K.$$

\subsection{Toroidal Lie algebras}
The universal central extension of $L_n(\mathfrak{g})$ can be expressed as $L_n(\mathfrak{g})\oplus \mathcal{Z}$, where $\mathcal{Z}$ represents the space of Kähler differentials. The vectors in $\mathcal{Z}$ are spanned by $\{t^{\bf{s} }K_i:\sum_{i=1}^{n}s_it^{\bf{s}}K_i=0, {\bf{s}}\in \Z^n, 1\leq i\leq n\}$. Further, we extend $L_n(\mathfrak{g})\oplus \mathcal{Z}$ as an algebra
$\mathcal{T}=$ $L_n(\mathfrak{g})\oplus \mathcal{Z}\oplus \mathfrak{D}$ where $\mathfrak{D}$ is the space spanned by the derivations $d_1,\dots,d_n$. Then the algebra $\mathcal{T}$ is called the toroidal Lie algebra with the following Lie bracket:
$$
[x\otimes t^{\bf{m}},y\otimes t^{\bf{n}}]=[x,y]\otimes t^{{\bf{m}}+{\bf{n}}}+(x|y)\sum_{i=1}^{n}m_i t^{{\bf{m}}+{\bf{n}}}K_i,$$
$$[L_n(\mathfrak{g}),\mathcal{Z}]=0=[\mathcal{Z},\mathcal{Z}]=[\mathfrak{D},\mathfrak{D}],
$$
$$[d_i,x\otimes t^{\bf{m}}]=m_ix\otimes t^{\bf{m}}, [d_i,t^{\bf{m}}K_j]=m_it^{\bf{m}}K_j,
$$
where $x,y\in \mathfrak{g}$, ${\bf{m}},{\bf{n}}\in \Z^n$, and $1\leq i,j\leq n$.

\subsection{Twisted toroidal Lie algebras} \label{sec5}
Define an endomorphism $\Tilde{\mu}$ of $L_n(\mathfrak{g})$ using the diagram automorphism $\mu$ of $\mathfrak{g}$ by:

$$
\Tilde{\mu}(x\otimes t^{\bf{s}})=\xi^{-s_1}\mu(x)\otimes t^{\bf{s}}.
$$One can verify that $\Tilde{\mu}$ is an automorphism of $L_n(\mathfrak{g})$. Let $L_n(\mathfrak{g},\mu)$ be the set of fixed points of $\Tilde{\mu}$. Note that $L_n(\mathfrak{g},\mu)$ is a subalgebra of $L_n(\mathfrak{g})$. Moreover, $L_n(\mathfrak{g},\mu)$ has the following $\Z/r\Z$ grading:
$$
L_n(\mathfrak{g},\mu)=\bigoplus _{{j}={0}}^{{r-1}}L_n(\mathfrak{g},\mu)_j,
$$
where $L_n(\mathfrak{g},\mu)_j=\mathfrak{g}_j\otimes t_1^j\C[t_1^{\pm r}, t_2^{\pm1},\dots,t_n^{\pm1}]$.

The universal central extension of $L_n(\mathfrak{g},\mu)$ can be expressed as $L_n(\mathfrak{g},\mu)\oplus \mathcal{Z'}$, where $\mathcal{Z'}$ is the subspace of $\mathcal{Z}$ and the vectors in $\mathcal{Z'}$ are spanned by $\{t^{\bf{s} }K_i\in \mathcal{Z}:{\bf{s}}\in \Z^n, 1\leq i
\leq n, s_1\cong 0$ (mod $r) \}$. Further, we extend $L_n(\mathfrak{g},\mu )\oplus \mathcal{Z'}$ as an algebra
$\mathcal{T}(\mu)=$ $L_n(\mathfrak{g},\mu)\oplus \mathcal{Z'}\oplus \mathfrak{D}$ where $\mathfrak{D}$ is defined as above. Note that $\mathcal{T}(\mu)$ is a subalgebra of $\mathcal{T}$. $\mathcal{T}(\mu)$ is called the twisted toroidal Lie algebra. In fact

$$\mathcal{T}(\mu)= \bigoplus _{{j}={0}}^{{r-1}}\mathfrak{g}_j\otimes t_1^j\C[t_1^{\pm r}, t_2^{\pm1},\dots,t_n^{\pm1}]\oplus \underset{m_1\in \Z, \underline {m}\in \Z^{n-1}} {\sum_{i=1}^{n}} t_1^{rm_1}t^{{\underline {m}}}K_i \oplus \mathfrak{D}.$$
The following is the standard triangular decomposition of $\mathcal{T}(\mu)$:
$$
    \mathcal{T}(\mu)= \mathcal{T}(\mu)^-\oplus \mathcal{T}(\mu)^0\oplus \mathcal{T}(\mu)^+,$$
where
$$\begin{aligned}
\mathcal{T}(\mu)^+&=\mathfrak{n}_0^+\otimes \C[t_2^{\pm1},\dots,t_n^{\pm1}]\oplus \mathfrak{g}_0\otimes t_1^{r}\C[t_1^{r},t_2^{\pm1},\dots,t_n^{\pm1}] \\&\oplus \bigoplus _{{j}={1}}^{{r-1}}\mathfrak{g}_j\otimes t_1^{j}\C[t_1^{r},t_2^{\pm1},\dots,t_n^{\pm1}]
\oplus {\sum_{i=1}^{n}} t_1^{r}\C[t_1^{r},t_2^{\pm1},\dots,t_n^{\pm1}]K_i,\\
\mathcal{T}(\mu)^-&=\mathfrak{n}_0^-\otimes \C[t_2^{\pm1},\dots,t_n^{\pm1}]\oplus \mathfrak{g}_0\otimes t_1^{-r}\C[t_1^{-r},t_2^{\pm1},\dots,t_n^{\pm1}] \\&\oplus \bigoplus _{{j}={1}}^{{r-1}}\mathfrak{g}_j\otimes t_1^{j-r}\C[t_1^{-r},t_2^{\pm1},\dots,t_n^{\pm1}]
\oplus {\sum_{i=1}^{n}} t_1^{-r}\C[t_1^{-r},t_2^{\pm1},\dots,t_n^{\pm1}]K_i,\\
\mathcal{T}(\mu)^0&= \mathfrak{h}_0\otimes\C[t_2^{\pm1},\dots,t_n^{\pm1}]\oplus {\sum_{i=1}^{n}}\C[t_2^{\pm1},\dots,t_n^{\pm1}]K_i\oplus \mathfrak{D}.
\end{aligned}$$
Let $\mathfrak{h}_{\mathcal{T}(\mu)}=\mathfrak{h}_0\oplus{\sum_{i=1}^{n}}\C K_i\oplus{\sum_{i=1}^{n}}\C d_i$. Then $\mathfrak{h}_{\mathcal{T}(\mu)}$ is a Cartan subalgebra of $\mathcal{T}(\mu)$. For $i\in I_0$, we extend $\alpha_i\in\mathfrak{h}_0^*$ to an element of $\mathfrak{h}_{\mathcal{T}(\mu)}^*$ by setting $\alpha_i(K_j)=\alpha_i(d_j)=0$ for $1\leq j\leq n$. For $1\leq i\leq n$, define $\delta_i\in \mathfrak{h}_{\mathcal{T}(\mu)}^*$ and $\gamma_i\in \mathfrak{h}_{\mathcal{T}(\mu)}^*$ by
$$
\begin{aligned}
\delta_i(\alpha_j^\vee)=0,&\,\,\,\,\,\gamma_i(\alpha_j^\vee)=0,\,\, \text{for all $j\in I_0,$}\\
\delta_i(K_j)=0,&\,\,\,\,\gamma_i(K_j)=\delta_{ij},\,\, \text{for all $1\leq j\leq n,$}\\
\delta_i(K_j)=\delta_{ij},&\,\,\,\,\gamma_i(K_j)=0,\,\, \text{for all $1\leq j\leq n.$}
\end{aligned}$$ Then it is easy to check that $\alpha_i$, $\delta_j$, $\gamma_j$ for $i\in I_0$, $1\leq j\leq n$ form a $\C$- basis of $\mathfrak{h}_{\mathcal{T}(\mu)}^*$. The nondegenerate invariant symmetric bilinear form $(\cdot|\cdot)$ on $\mathfrak{h}_0$ induces a nondegenerate form  $<\cdot|\cdot>$ on $\mathfrak{h}_{\mathcal{T}(\mu)}^*$ defined as
$$<\alpha_i|\alpha_j> \,\,=(\alpha_i|\alpha_j)\,\,\text{for all $i,j\in I_0$},$$ 
$$<\alpha_i|\delta_j>\,\,=\,\,<\alpha_i|\gamma_j>\,\,=0\,\,\text{for all $i\in I_0, 1\leq j\leq n,$}
$$
$$<\delta_i|\gamma_j>\,\,=\,\,\delta_{ij},\,\, \text{and}\,\,<\delta_i|\delta_j>\,\,=\,\,<\gamma_i|\gamma_j>\,\,=0\,\, \text{for all $ 1\leq i, j\leq n$}.$$
For $\bf{m}\in\Z^n$, define $\delta_{\bf{m}}\coloneqq{\sum_{i=1}^{n}}m_i\delta_i\in\mathfrak{h}_{\mathcal{T}(\mu)}^*$. With respect to $\mathfrak{h}_0$, let $R_0$, $R_0(+)$, $R_0(-)$, and $R_0^s$ denote the set of all roots, positive roots, negative roots, and short roots of $\mathfrak{g}_0$, respectively. With respect to $\mathfrak{h}_{\mathcal{T}(\mu)}$, let $R^{\mu}$, $R^{\mu}(+)$, and $R^{\mu}(-)$ denote the set of all roots, positive roots, and negative roots of $\mathcal{T}(\mu)$, respectively. Then for $\mathfrak{g}=A_{2\ell-1},\, D_{\ell+1},\, D_4$, we have
$$\begin{aligned}
R^{\mu}&=\{\alpha+\delta_{\bf{m}}\,\,:\,\, \alpha\in R_0,\,\, {\bf{m}}\in\Z^n,\,\, m_1\cong 0\,(\text{mod r})\}\\&\cup\{\alpha+\delta_{\bf{m}}\,\, :\,\, \alpha\in R_0^s,\,\, {\bf{m}}\in\Z^n,\,\, m_1\ncong 0\,\,(\text{mod r})\}\\
&\cup\{\delta_{\bf{m}}\,\,:\,\,{\bf{m}}\in\Z^n,\,\, {\bf{m}}\neq {\bf{0}}\},
\end{aligned}$$
$$\begin{aligned}
R^{\mu}(+)&=\{\alpha+\delta_{\bf{m}}\,\,:\,\, \alpha\in R_0(+),\,\, {\bf{m}}\in\Z^n,\,\, m_1=0\}\\&\cup\{\alpha+\delta_{\bf{m}}\,\,:\,\, \alpha\in R_0,\,\, {\bf{m}}\in\Z^n,\,\,m_1>0,\,\, m_1\cong 0\,(\text{mod r})\}\\&\cup\{\alpha+\delta_{\bf{m}}\,\, :\,\, \alpha\in R_0^s,\,\, {\bf{m}}\in\Z^n,\,\,\,\,m_1>0, m_1\ncong 0\,\,(\text{mod r})\}\\
&\cup\{\delta_{\bf{m}}\,\,:\,\,{\bf{m}}\in\Z^n,\,\, m_1>0\},
\end{aligned}$$
$$\begin{aligned}
    R^{\mu}(-)&=-R^{\mu}(+).
\end{aligned}$$
For $\mathfrak{g}=A_{2\ell}$, we have
$$\begin{aligned}
R^{\mu}&=\{\alpha+\delta_{\bf{m}}\,\,:\,\, \alpha\in R_0,\,\, {\bf{m}}\in\Z^n\}\\&\cup\{r\alpha+\delta_{\bf{m}}\,\, :\,\, \alpha\in R_0^s,\,\, {\bf{m}}\in\Z^n,\,\, m_1\ncong 0\,\,(\text{mod r})\}\\
&\cup\{\delta_{\bf{m}}\,\,:\,\,{\bf{m}}\in\Z^n,\,\, {\bf{m}}\neq {\bf{0}}\},
\end{aligned}$$
$$\begin{aligned}
R^{\mu}(+)&=\{\alpha+\delta_{\bf{m}}\,\,:\,\, \alpha\in R_0(+),\,\, {\bf{m}}\in\Z^n,\,\, m_1=0\}\\&\cup\{\alpha+\delta_{\bf{m}}\,\,:\,\, \alpha\in R_0,\,\, {\bf{m}}\in\Z^n,\,\,m_1>0\}\\&\cup\{r\alpha+\delta_{\bf{m}}\,\, :\,\, \alpha\in R_0^s,\,\, {\bf{m}}\in\Z^n,\,\,\,\,m_1>0, m_1\ncong 0\,\,(\text{mod r})\}\\
&\cup\{\delta_{\bf{m}}\,\,:\,\,{\bf{m}}\in\Z^n,\,\, m_1>0\},
\end{aligned}$$
$$\begin{aligned}
    R^{\mu}(-)&=-R^{\mu}(+).
\end{aligned}$$

Define $R^\mu(0)\coloneqq\{\delta_{\bf{m}}\,\,:\,\,{\bf{m}}\in\Z^n,\,\, m_1=0,\,\,{\bf{m}}\neq {\bf{0}}\}$, then $$R^{\mu} = R^{\mu}(+)\cup R^\mu(0)\cup R^{\mu}(-).$$ A root $\alpha\in R^\mu$ is called an imaginary root if $<\alpha|\alpha>=0$ otherwise it is called a real root. For $1\leq i\leq n$, we set $\alpha_{\ell+i}=\delta_i-\theta^0$, then the set $\{\alpha_i\,\,:\,\,1\leq i\leq \ell+n\}$  forms a simple root system for $\mathcal{T}({\mu})$.

\subsection{\texorpdfstring{Twisted affine subalgebra of $\mathcal{T}(\mu)$}{}}
The twisted toroidal Lie algebra contains a subalgebra isomorphic to the twisted affine Lie algebra which is given by:
$$\mathcal{T}_{\mathrm{aff}}(\mu)= \bigoplus _{{j}={0}}^{{r-1}}\mathfrak{g}_j\otimes t_1^j\C[t_1^{\pm r}] \oplus \C K_1\oplus\C d_1.
$$
The following is the standard triangular decomposition of $\mathcal{T}_{\mathrm{aff}}(\mu)$:
$$\mathcal{T}_{\mathrm{aff}}(\mu)=\mathcal{T}_{\mathrm{aff}}(\mu)^-\oplus \mathcal{T}_{\mathrm{aff}}(\mu)^0\oplus \mathcal{T}_{\mathrm{aff}}(\mu)^+
,$$
where
$$
\begin{aligned}
    \mathcal{T}_{\mathrm{aff}}(\mu)^+&=\mathfrak{n}_0^+\oplus \mathfrak{g}_0\otimes t_1^{r}\C[t_1^{r}] \oplus \bigoplus _{{j}={1}}^{{r-1}}\mathfrak{g}_j\otimes t_1^{j}\C[t_1^{r}],\\
\mathcal{T}_{\mathrm{aff}}(\mu)^-&=\mathfrak{n}_0^-\oplus \mathfrak{g}_0\otimes t_1^{-r}\C[t_1^{-r}] \oplus \bigoplus _{{j}={1}}^{{r-1}}\mathfrak{g}_j\otimes t_1^{j-r}\C[t_1^{-r}],\\
\mathcal{T}_{\mathrm{aff}}(\mu)^0&=\mathfrak{h}_0\oplus \C K_1\oplus\C d_1
.\end{aligned}$$

Let $\mathfrak{h}_{\mathcal{T}_{\mathrm{aff}}(\mu)}=\mathcal{T}_{\mathrm{aff}}(\mu)^0$. Then $\mathfrak{h}_{\mathcal{T}_{\mathrm{aff}}(\mu)}$ is a Cartan subalgebra of $\mathcal{T}_{\mathrm{aff}}(\mu)$. With respect to $\mathfrak{h}_{\mathcal{T}_{\mathrm{aff}}(\mu)}$,
let $R_{\mathrm{aff}}^{\mu}$, $R_{\mathrm{aff}}^{\mu}(+)$, and $R_{\mathrm{aff}}^{\mu}(-)$ denotes the set of all roots, positive roots, and negative roots of $\mathcal{T}_{\mathrm{aff}}(\mu)$, respectively. Then for $\mathfrak{g}=A_{2\ell-1},\, D_{\ell+1},\, D_4$, we have
$$\begin{aligned}
R_{\mathrm{aff}}^{\mu}&=\{\alpha+m\delta_{1}\,\,:\,\, \alpha\in R_0,\,\, {m}\in\Z,\,\, m\cong 0\,(\text{mod r})\}\\&\cup\{\alpha+m\delta_{1}\,\, :\,\, \alpha\in R_0^s,\,\, {m}\in\Z,\,\, m_1\ncong 0\,\,(\text{mod r})\}\\
&\cup\{m\delta_{1}\,\,:\,\,{m}\in\Z,\,\, {m}\neq {0}\},
\end{aligned}$$
$$\begin{aligned}
R_{\mathrm{aff}}^{\mu}(+)&=\{\alpha+m\delta_{1}\,\,:\,\, \alpha\in R_0,\,\, m\in\Z_{>0},\,\, m\cong 0\,(\text{mod r})\}\\&\{\alpha+m\delta_{1}\,\,:\,\, \alpha\in R_0^s,\,\,  m\in\Z_{>0},\,\, m\ncong 0\,(\text{mod r})\}\\
&\cup\{m\delta_{1}\,\,:\,\,{m}\in\Z_{>0}\}\cup R_0(+),
\end{aligned}$$
$$\begin{aligned}
    R_{\mathrm{aff}}^{\mu}(-)&=-R_{\mathrm{aff}}^{\mu}(+).
\end{aligned}$$
For $\mathfrak{g}=A_{2\ell}$, we have
$$\begin{aligned}
R^{\mu}&=\{\alpha+m\delta_{1}\,\,:\,\, \alpha\in R_0,\,\, {m}\in\Z\}\\&\cup\{r\alpha+m\delta_{1}\,\, :\,\, \alpha\in R_0^s,\,\, m\in\Z,\,\, m\ncong 0\,\,(\text{mod r})\}\\
&\cup\{m\delta_{1}\,\,:\,\,{m}\in\Z,\,\, {m}\neq {0}\},
\end{aligned}$$
$$\begin{aligned}
R^{\mu}(+)&=\{\alpha+m\delta_{1}\,\,:\,\, \alpha\in R_0,\,\, m\in\Z_{>0}\}\\&\cup\{r\alpha+m\delta_{1}\,\, :\,\, \alpha\in R_0^s,\,\, m\in\Z_{>0}, m\ncong 0\,\,(\text{mod r})\}\\
&\cup\{m\delta_{1}\,\,:\,\,m\in\Z_{>0}\}\cup R_0(+),
\end{aligned}$$
$$\begin{aligned}
    R^{\mu}(-)&=-R^{\mu}(+).
\end{aligned}$$

Let $$\begin{aligned}
      e_0=&\begin{cases}
     f'_{{\theta}_0}\otimes t_1,&(A_{2\ell}),\\[.5em]
     f^{(1)}_{\theta_s}\otimes t_1={\sum_{j=0}^{r-1}}\mu^{j}(\xi^{r-j}f'_{\theta_{0}})\otimes t_1, &(A_{2\ell-1},D_{\ell+1}, D_4),
\end{cases}\\[.5em]
f_0=&
\begin{cases}
    e'_{{\theta}_0}\otimes t_1^{-1},&(A_{2\ell}),\\[.5em]
    e^{(r-1)}_{\theta_s}\otimes t_1^{-1}={\sum_{j=0}^{r-1}}\mu^{j}(\xi^{j}e'_{\theta_{0}})\otimes t_1^{-1}, &(A_{2\ell-1},D_{\ell+1}, D_4),
    \end{cases}\end{aligned}$$
    and
    $$h_0=[e_0,f_0].$$

 We set $\alpha_{0}=\delta_1-\theta^0$ and $\tilde{I}_0=I_0\cup \{0\}$, then the sets $\{\alpha_i\,\,:\,\,i\in \tilde{I}_0\}$ and $\{\alpha_0^\vee=h_0,\,\,\alpha_i^\vee\,\,:\,\,i\in {I}_0\}$ form a simple root and coroot system for $\mathcal{T}_{\mathrm{aff}}({\mu})$%{, where $a_0=2$ for $\mathfrak{g}=A_{2\ell}$ otherwise it is $1$}
 . Let $\Lambda_i$, $i\in \tilde{I}_0$ be the fundamental weights of $\mathcal{T}_{\mathrm{aff}}({\mu})$.

\section{\texorpdfstring{Presentation and some Properties of $\mathcal{T}(\mu)$}{}} \label{secpr}
Recently, in \cite{JMMishra}, the MRY presentation of $\mathcal{T}(\mu)$ for $n = 2$ has been obtained. In this section, we obtain the presentation of $\mathcal{T}(\mu)$ for any $n\geq 2$. Also, we introduce two subalgebras of $\mathcal{T}(\mu)$, namely $\mathcal{T}^+(\mu)$ and $\overline{\mathcal{T}}(\mu)$, where the former plays an important role in this paper. 

\subsection{\texorpdfstring{Presentation of $\mathcal{T}(\mu)$}{}}
Let us extend the Cartan matrix $A$ to $\Tilde{A}=(a_{i,j})_{i,j\in \tilde{I}_0}$ by defining:
 $$ \begin{cases}
       \begin{cases}
    a_{0,1}=-1, a_{1,0}=-2, & \text{for $\mathfrak{g}=A_{2\ell}$}.\\[.3em]
    a_{0,0}=2, a_{0,2}=a_{2,0}=-1, & \text{for $\mathfrak{g}=A_{2\ell-1}$}.\\[.3em]
    a_{0,1}=-2, a_{1,0}=-1, & \text{for $\mathfrak{g}=D_{\ell+1}$}.\\[.3em]
     a_{0,1}=a_{1,0}=-1, & \text{for $\mathfrak{g}=D_{4}$}.
  \end{cases}\\[3em]
  a_{0,j}=a_{j,0}=0, & \text{otherwise for all types.}
  \end{cases}$$
  Let $\mathfrak{t(g)}$ be the Lie algebra over $\C$ generated by the symbols $$ \delta_{\underline{r}}(\underline{s}), \,
 \alpha_i(\underline{k}),\, X(\pm \alpha_i,\underline{k}),\,d_j\,(i\in \tilde{I}_0,\, \underline{r}, \underline{s}, \underline{k} \in \Z^{n-1},\, j=1,\dots
,n),
  $$ and satisfying the following relations: \begin{itemize}\setlength\itemsep{1em}
      \item \begin{enumerate}[label=(\roman*)]\setlength\itemsep{.7em}
          \item $\delta_{\underline{r}}(\underline{s})+\delta_{\underline{k}}(\underline{s})=\delta_{\underline{r}+\underline{k}}(\underline{s})$;
          \item $\delta_{\underline{r}}(\underline{r})=0$;
          \item $[\delta_{\underline{r}}(\underline{s}),\delta_{\underline{k}}(\underline{l})]=[\delta_{\underline{r}}(\underline{s}), \alpha_i(\underline{k})]= [\delta_{\underline{r}}(\underline{s}),X(\pm \alpha_i,\underline{k})]=0$;
          \item $[d_1,\delta_{\underline{r}}(\underline{s})]=0, [d_j,\delta_{\underline{r}}(\underline{s})]=s_j\delta_{\underline{r}}(\underline{s}),$ \\[.5em] where $j=2,\dots,n.$
      \end{enumerate}
      \item $[\alpha_0(\underline{k}),\alpha_0(\underline{l})]=\begin{cases}
          2\delta_{\underline{k}}(\underline{k}+\underline{l}), &(A_{2\ell}),\\[.5em]
          2r\delta_{\underline{k}}(\underline{k}+\underline{l}), &(A_{2\ell-1}, D_{\ell+1},D_4).
      \end{cases}$
      \item $[\alpha_0(\underline{k}),\alpha_j(\underline{l})]=\begin{cases}
          ra_{0,j}\delta_{\underline{k}}(\underline{k}+\underline{l}), &(A_{2\ell}, A_{2\ell-1}, D_4),\\[.5em]
          a_{0,j}\delta_{\underline{k}}(\underline{k}+\underline{l}), &( D_{\ell+1}),
      \end{cases}$ \\[.5em] where $j\in I_0$.
      \item $[\alpha_i(\underline{k}),\alpha_j(\underline{l})]=\begin{cases}
          ra_{i,j}\delta_{\underline{k}}(\underline{k}+\underline{l}), &(A_{2\ell}, A_{2\ell-1}, D_4),\\[.5em]
          a_{i,j}\delta_{\underline{k}}(\underline{k}+\underline{l}), &( D_{\ell+1}),
      \end{cases}$\\[.5em]
      where $i,j\in I_0$ with $i\leq j$ and $(i,j)\neq (\ell-1,\ell), (\ell,\ell)$.
      \item $[\alpha_{\ell-1}(\underline{k}),\alpha_\ell(\underline{l})]=\begin{cases}
          4a_{\ell-1,\ell}\delta_{\underline{k}}(\underline{k}+\underline{l}), &(A_{2\ell}),\\[.5em]
          a_{\ell-1,\ell}\delta_{\underline{k}}(\underline{k}+\underline{l}), &(A_{2\ell-1}, D_4),\\[.5em]
          2a_{\ell-1,\ell}\delta_{\underline{k}}(\underline{k}+\underline{l}), &(D_{\ell+1}).
      \end{cases}$
      \item $[\alpha_{\ell}(\underline{k}),\alpha_\ell(\underline{l})]=\begin{cases}
          8\delta_{\underline{k}}(\underline{k}+\underline{l}), &(A_{2\ell}),\\[.5em]
          2\delta_{\underline{k}}(\underline{k}+\underline{l}), &(A_{2\ell-1}, D_4),\\[.5em]
          4\delta_{\underline{k}}(\underline{k}+\underline{l}), &(D_{\ell+1}).
      \end{cases}$
      \item $[\alpha_{i}(\underline{k}),X(\pm \alpha_j,\underline{l})]=\pm a_{i,j}X(\pm \alpha_j,\underline{k}+\underline{l})$,\\[.5em]
      where $i,j\in \tilde{I}_0$.
      \item 
      $[X(\pm \alpha_i,\underline{k}),X(\pm \alpha_i,\underline{l})]=0,$\\[.5em]
      where $i\in\tilde{I}_0$.
      \item $[X( \alpha_i,\underline{k}),X(-\alpha_j,\underline{l})]\\[.5em]
    =\begin{cases}
    {\delta}_{i,j}\{\alpha_i(\underline{k}+\underline{l})+(r(1+\delta_{i,\ell}(r-1))-\delta_{i,0}(r-1))\delta_{\underline{k}}(\underline{k}+\underline{l})\}, & (A_{2\ell}),\\[.5em]
        {\delta}_{i,j}\{\alpha_i(\underline{k}+\underline{l})+(r-\delta_{i,\ell}(r-1))\delta_{\underline{k}}(\underline{k}+\underline{l})\}, & (A_{2\ell-1}, D_4),\\[.5em]
        {\delta}_{i,j}\{\alpha_i(\underline{k}+\underline{l})+(1+(\delta_{i,0}+\delta_{i,\ell})(r-1))\delta_{\underline{k}}(\underline{k}+\underline{l})\}, & (D_{\ell+1}),
    \end{cases}$\\[.5em]
     where $i,j\in\tilde{I}_0$.
    \item $\text{ad}X( \pm\alpha_i,\underline{k})X(\pm\alpha_j,\underline{l})=0,$\\[.5em]
    where $i,j\in\tilde{I}_0 \, \text{with} \, i\neq j\, \text{and} \, a_{i,j}=0.$
    \item $\text{ad}X( \pm\alpha_i,\underline{k})\text{ad}X( \pm\alpha_i,\underline{l})X(\pm\alpha_j,\underline{s})=0,$\\[.5em]
    where $i,j\in\tilde{I}_0 \, \text{with} \, i\neq j\, \text{and} \, a_{i,j}=-1.$
    \item $\text{ad}X( \pm\alpha_i,\underline{k})\text{ad}X( \pm\alpha_i,\underline{l})\text{ad}X( \pm\alpha_i,\underline{s})X(\pm\alpha_j,\underline{r})=0,$\\[.5em]
    where $i,j\in\tilde{I}_0 \, \text{with} \, i\neq j\, \text{and} \, a_{i,j}=-2.$
    
     \item $\text{ad}X( \pm\alpha_i,\underline{k})\text{ad}X( \pm\alpha_i,\underline{l})\text{ad}X( \pm\alpha_i,\underline{s})\text{ad}X( \pm\alpha_i,\underline{r})X(\pm\alpha_j,\underline{u})=0,$\\[.5em]
    where $i,j\in\tilde{I}_0 \, \text{with} \, i\neq j\, \text{and} \, a_{i,j}=-3.$

    \item $[d_1,\alpha_i(\underline{k})]=0, [d_1,X( \pm\alpha_i,\underline{k})]=\pm\delta_{i,0}X( \pm\alpha_i,\underline{k}),$\\[.5em]
    where $i\in \tilde{I}_0$.
    \item $[d_j,\alpha_i(\underline{k})]=k_j\alpha_i(\underline{k}),[d_j,X( \pm\alpha_i,\underline{k})]=k_jX( \pm\alpha_i,\underline{k}),$\\[.5em]
    where $i\in \tilde{I}_0$ and $j=2,\dots,n.$
    \item $[d_i,d_j]=0,$\\[.5em]
    where $i,j=1,\dots,n.$
    \end{itemize}
    \begin{thm}
     As a Lie algebra, $\mathfrak{t(g)}$ and $\mathcal{T}(\mu)$ are isomorphic. The isomorphism $\phi:\mathfrak{t(g)}\rightarrow\mathcal{T}(\mu)$ is explicitly given by:\vspace{-7em}
      $$\left.\begin{aligned}
 \delta_{\underline{r}}(\underline{s}) &\mapsto {\sum_{i=2}^{n}}r_it^{\underline{s}}K_i\\[.5em]
\alpha_0(\underline{k})  &\mapsto -h'_{\theta_{0}}\otimes t^{\underline{k}}+t^{\underline k}K_1\\[.5em]
  \alpha_i(\underline{k})&\mapsto \Big(1+\delta_{i,\ell}\Big) h^{(0)}_i\otimes t^{\underline k} \\[.5em]
  X(\alpha_0,\underline{k})&\mapsto -f'_{{\theta}_0}\otimes t_1t^{\underline k}\\[.5em]
  X(-\alpha_0,\underline{k})&\mapsto -e'_{{\theta}_0}\otimes t_1^{-1}t^{\underline k}\\[.5em]
  X(\alpha_i,\underline{k})&\mapsto \Big(1+\delta_{i,\ell}\Big(\sqrt{2}-1\Big)\Big) e^{(0)}_i\otimes t^{\underline k}\\[.5em]
  X(-\alpha_i,\underline{k})&\mapsto \Big(1+\delta_{i,\ell}\Big(\sqrt{2}-1\Big)\Big) f^{(0)}_i\otimes t^{\underline k}\\[.5em]
  d_j&\mapsto d_j
\end{aligned}\right\}\text{for $\mathfrak{g}=A_{2\ell},$}$$
$$\left.\begin{aligned}
 \delta_{\underline{r}}(\underline{s}) &\mapsto {\sum_{i=2}^{n}}r_it^{\underline{s}}K_i \\[.5em]
\alpha_0(\underline{k})  &\mapsto {\sum_{j=0}^{r-1}}-\mu^{j}(h'_{\theta_{0}})\otimes t^{\underline{k}}+rt^{\underline k}K_1\\[.5em]
\alpha_i(\underline{k})&\mapsto \Big(1-\delta_{i,\mu(i)}\Big(1-\frac{1}{r}\Big)\Big) h^{(0)}_i\otimes t^{\underline k} \\[.5em]
  X(\alpha_0,\underline{k})&\mapsto {\sum_{j=0}^{r-1}}-\mu^{j}(\xi^{r-j}f'_{\theta_{0}})\otimes t_1t^{\underline k}\\[.5em]
  X(-\alpha_0,\underline{k})&\mapsto {\sum_{j=0}^{r-1}}-\mu^{j}(\xi^{j}e'_{\theta_{0}})\otimes t_1^{-1}t^{\underline k}\\[.5em]
  X(\alpha_i,\underline{k})&\mapsto  \Big(1-\delta_{i,\mu(i)}\Big(1-\frac{1}{r}\Big)\Big) e^{(0)}_i\otimes t^{\underline k}\\[.5em]
  X(-\alpha_i,\underline{k})&\mapsto  \Big(1-\delta_{i,\mu(i)}\Big(1-\frac{1}{r}\Big)\Big) f^{(0)}_i\otimes t^{\underline k}\\[.5em]
  d_j&\mapsto d_j
\end{aligned}\right\}\text{for $\mathfrak{g}=A_{2\ell-1},\, D_{\ell+1},\, D_4$}.$$
  \end{thm}
  \begin{proof}
       It is sufficient to prove that $\phi$ preserves the defining relations. We prove only $ [\phi(X(\alpha_0,\underline{k})), \phi(X(-\alpha_0,\underline{l}))]=\phi([X(\alpha_0,\underline{k}), X(-\alpha_0,\underline{l})]$ and the other relations will follow with a similar calculation.
For $A_{2\ell}$,
\vspace{.01em} $$\begin{aligned}
[\phi(X(\alpha_0,\underline{k})),& \phi(X(-\alpha_0,\underline{l}))]=[-f'_{{\theta}_0}\otimes t_1t^{\underline k},-e'_{{\theta}_0}\otimes t_1^{-1}t^{\underline l}]\\
&=-h'_{{\theta}_0}\otimes t^{\underline k+\underline{l}} + (f'_{{\theta}_0}| e'_{{\theta}_0})\Big(t^{\underline k}K_1+{\sum_{p=2}^{n}}k_it^{\underline k+\underline l}K_p\Big)\\
&=-h'_{{\theta}_0}\otimes t^{\underline k+\underline{l}} +t^{\underline k}K_1+{\sum_{p=2}^{n}}k_it^{\underline k+\underline l}K_p\\
&=\phi(\alpha_0(\underline{k}+\underline{l})+\delta_{\underline{k}}(\underline{k}+\underline{l}))\\
&=\phi([X(\alpha_0,\underline{k}), X(-\alpha_0,\underline{l})].
\end{aligned}$$
       For $A_{2\ell-1},\, D_{\ell+1},\, D_4,$\vspace{-.25em}
$$\begin{aligned}
[\phi(X(\alpha_0,\underline{k})),& \phi(X(-\alpha_0,\underline{l}))] =\Big[{\sum_{i=0}^{r-1}}-\mu^{i}(\xi^{r-i}f'_{\theta_{0}})\otimes t_1t^{\underline k}, {\sum_{j=0}^{r-1}}-\mu^{j}(\xi^{j}e'_{\theta_{0}})\otimes t_1^{-1}t^{\underline l}\Big]\\
&={\sum_{i,\,j=0}^{r-1}}\xi^{r-i+j}\Big(\Big[\mu^{i}(f'_{\theta_{0}}),\mu^{j}(e'_{\theta_{0}})\Big]\otimes t^{\underline k+\underline l}+ (\mu^{i}(f'_{\theta_{0}})| \mu^{j}(e'_{\theta_{0}}))
{\sum_{p=2}^{n}}k_it^{\underline k+\underline l}K_p \Big).
\end{aligned}$$
Direct calculation gives us the following:
$$\Big[\mu^{i}(f'_{\theta_{0}}),\mu^{j}(e'_{\theta_{0}})\Big]=\begin{cases}
    -h'_{\theta_0},&\text{if}\,\, i=j=0 \, (A_{2\ell-1},\, D_{\ell+1},\, D_4), \\[.5em]
    -\mu(h'_{\theta_0}), &\text{if}\,\, i=j=1\, (A_{2\ell-1},D_{\ell+1}, D_4),\\[.5em]
    -\mu^2(h'_{\theta_0}), &\text{if}\,\, i=j=2\, (D_4),\\[.5em]
    0,&\text{if}\,\, i\neq j\, (A_{2\ell-1},D_{\ell+1}, D_4),
\end{cases}$$
and
$$(\mu^{i}(f'_{\theta_{0}})| \mu^{j}(e'_{\theta_{0}}))=\begin{cases}
    r,&\text{if}\,\, i=j\, (A_{2\ell-1},D_{\ell+1}, D_4),\\[.5em]
    0,&\text{if}\,\, i\neq j\, (A_{2\ell-1},D_{\ell+1}, D_4).
\end{cases}
$$

Hence,
$$\begin{aligned}
[\phi(X(\alpha_0,\underline{k})),& \phi(X(-\alpha_0,\underline{l}))]={\sum_{j=0}^{r-1}}-\mu^{j}(h'_{\theta_{0}})\otimes t^{\underline{k}+\underline{l}}+r\Big(t^{\underline k}K_1+{\sum_{p=2}^{n}}k_it^{\underline k+\underline l}K_p\Big)\\
&=\phi(\alpha_0(\underline{k}+\underline{l})+r \delta_{\underline{k}}(\underline{k}+\underline{l}))\\
&=\phi([X(\alpha_0,\underline{k}), X(-\alpha_0,\underline{l})].
\end{aligned}$$
  \end{proof}
   \subsection{\texorpdfstring{Some subalgebras of $\mathcal{T}(\mu)$}{}}
 Let $\mathcal{T}^+(\mu)$ be the subalgebra of $\mathcal{T}(\mu)$ generated by $\{e_ {i,\underline k}= e^{(0)}_i \otimes t^{\underline k},\, f_ {i,\underline k}= f^{(0)}_i \otimes t^{\underline k},\, d_1: i\in {I_0},\, \underline k\in \Z^{n-1}_{\geq 0}\}\cup \{e_{0,\underline k}= e_0 \otimes t^{\underline k},\, f_{0,\underline k}= f_0 \otimes t^{\underline k},\, \underline k\in \Z^{n-1}_{\geq 0}\}$. We have the following explicit form of $\mathcal{T}^+(\mu)$:
$$\begin{aligned}
    \mathcal{T}^+(\mu)= \bigoplus _{{j}={0}}^{{r-1}}\mathfrak{g}_j\otimes&\,\, t_1^{j}\C[t_1^{\pm r},t_2,\dots,t_n] \oplus \underset{k_1\in \Z,\, \underline {k}\in \Z_{\ge 0}^{n-1}} {\sum} \C t_1^{rk_1}t^{{\underline {k}}}K_1\\ &\oplus \underset{k_1\in \Z,\, \underline {k}\in \Z_{\geq0}^{n-1}, k_i\geq 1} {\sum_{i=2}^{n}} \C t_1^{rk_1}t^{{\underline {k}}}K_i \oplus \C d_1.
\end{aligned}
$$
Let $\overline{\mathcal{T}}(\mu)$ be the subalgebra of $\mathcal{T}(\mu)$ generated by $\{e_ {i,\underline k}= e^{(0)}_i \otimes t^{\underline k},\, f_ {i,\underline k}= f^{(0)}_i \otimes t^{\underline k},\,  d_1: i\in I_0,\,\underline k \in \Z^{n-1}\}\cup \{e_{0,\underline k}= e_0 \otimes t^{\underline k},\, f_{0,\underline k}= f_0 \otimes t^{\underline k},\, \underline k\in \Z^{n-1}\}$. We have the following explicit form of $\overline{\mathcal{T}}(\mu)$:
$$\overline{\mathcal{T}}(\mu)= \bigoplus _{{j}={0}}^{{r-1}}\mathfrak{g}_j\otimes t_1^{j}\C[t_1^{\pm r},t_2^{\pm 1},\dots,t_n^{\pm 1}] \oplus \underset{k_1\in \Z,\, \underline {k}\in \Z^{n-1}} {\sum_{i=1}^{n}} \C t_1^{rk_1}t^{{\underline {k}}}K_i \oplus \C d_1.
$$Let $\overline{\mathfrak{h}}_{\mathcal{T}(\mu)}= \mathcal{T}(\mu)^0\bigcap \overline{\mathcal{T}}(\mu)$. Then we have
$$\overline{\mathfrak{h}}_{\mathcal{T}(\mu)}=\mathfrak{h}_0\otimes\C[t_2^{\pm1},\dots,t_n^{\pm1}]\oplus {\sum_{i=1}^{n}}\C[t_2^{\pm1},\dots,t_n^{\pm1}]K_i\oplus \C d_1.
$$
\subsection{\texorpdfstring{Automorphisms of $\mathcal{T}(\mu)$}{}}

Consider the following Lie algebra automorphisms of $\mathcal{T}(\mu)$ which were first defined in \cite{FIBKVG} in order to study representation of affine Kac-Moody algebras.$$
\Psi_0 \coloneqq\mathrm{exp\, ad}(e_0)\mathrm{exp\, ad}(-f_0)\mathrm{exp\, ad}(e_0),$$
$$\Psi_{\theta_s}\coloneqq\mathrm{exp\, ad}(e^{(0)}_{\theta_{s}})\mathrm{exp\, ad}(-f^{(0)}_{\theta_{s}})\mathrm{exp\, ad}(e^{(0)}_{\theta_{s}}).
$$
\begin{lemma} \label{RP1} For $\mathfrak{g}=A_{2\ell-1},\, D_{\ell+1},\, D_4,$ we have, $$\Psi_0\Psi_{\theta_s}(e^{(j)}_{\theta_s}\otimes t_1^{rm_1+j}t^{\underline{m}})=e_{\theta_s}^{(r+j-2)}\otimes t_1^{rm_1+j-2}t^{\underline{m}},$$ for all $m_1\in \Z$, $\underline{m}\in\Z^{n-1}$, and $0\leq j\leq r-1.$

\end{lemma}
\begin{proof}\noindent
We have the following two claims, and using them, we get the desired result:

\noindent {\bf Claim 1 :} $ \Psi_{\theta_s}(e^{(j)}_{\theta_s}\otimes t_1^{rm_1+j}t^{\underline{m}})=-f_{\theta_s}^{(j)}\otimes t_1^{rm_1+j}t^{\underline{m}}$, for all $m_1\in \Z$, $\underline{m}\in\Z^{n-1}$, and $0\leq j\leq r-1.$

   \noindent Direct calculation gives us the following:
$$\begin{aligned}\mathrm{exp\, ad}(e^{(0)}_{\theta_{s}})(e^{(j)}_{\theta_s}\otimes t_1^{rm_1+j}t^{\underline{m}})&=e^{(j)}_{\theta_s}\otimes t_1^{rm_1+j}t^{\underline{m}},\\
\mathrm{exp\, ad}(-f^{(0)}_{\theta_{s}})(e^{(j)}_{\theta_s}\otimes t_1^{rm_1+j}t^{\underline{m}})&=e^{(j)}_{\theta_s}\otimes t_1^{rm_1+j}t^{\underline{m}}-[f^{(0)}_{\theta_{s}},e^{(j)}_{\theta_s}]\otimes t_1^{rm_1+j}t^{\underline{m}}\\
&\,\,\,\,\,\,\,\,-f^{(j)}_{\theta_s}\otimes t_1^{rm_1+j}t^{\underline{m}},\\
\mathrm{exp\, ad}(e^{(0)}_{\theta_{s}})([f^{(0)}_{\theta_{s}},e^{(j)}_{\theta_s}]\otimes t_1^{rm_1+j}t^{\underline{m}})&=[f^{(0)}_{\theta_{s}},e^{(j)}_{\theta_s}]\otimes t_1^{rm_1+j}t^{\underline{m}}+2e^{(j)}_{\theta_s}\otimes t_1^{rm_1+j}t^{\underline{m}},\\
\mathrm{exp\, ad}(e^{(0)}_{\theta_{s}})(f^{(j)}_{\theta_s}\otimes t_1^{rm_1+j}t^{\underline{m}})&=f^{(j)}_{\theta_s}\otimes t_1^{rm_1+j}t^{\underline{m}}+[e^{(0)}_{\theta_{s}},f^{(j)}_{\theta_s}]\otimes t_1^{rm_1+j}t^{\underline{m}}\\
&\,\,\,\,\,\,\,\,-e^{(j)}_{\theta_s}\otimes t_1^{rm_1+j}t^{\underline{m}}.
\end{aligned}
$$

\noindent Hence,
$$\begin{aligned}\Psi_{\theta_s}(e^{(j)}_{\theta_s}\otimes t_1^{rm_1+j}t^{\underline{m}})&=e^{(j)}_{\theta_s}\otimes t_1^{rm_1+j}t^{\underline{m}}-[f^{(0)}_{\theta_{s}},e^{(j)}_{\theta_s}]\otimes t_1^{rm_1+j}t^{\underline{m}}-2e^{(j)}_{\theta_s}\otimes t_1^{rm_1+j}t^{\underline{m}}\\&\,\,\,\,\,\,\,\,-f^{(j)}_{\theta_s}\otimes t_1^{rm_1+j}t^{\underline{m}}-[e^{(0)}_{\theta_{s}},f^{(j)}_{\theta_s}]\otimes t_1^{rm_1+j}t^{\underline{m}}+e^{(j)}_{\theta_s}\otimes t_1^{rm_1+j}t^{\underline{m}}\\&=-f^{(j)}_{\theta_s}\otimes t_1^{rm_1+j}t^{\underline{m}}.\end{aligned} $$

\noindent{\bf Claim 2 :} $\Psi_0(f_{\theta_s}^{(j)}\otimes t_1^{rm_1+j}t^{\underline{m}})=-e_{\theta_s}^{(r-j)}\otimes t_1^{rm_1+j-2}t^{\underline{m}}$, for all $m_1\in \Z$, $\underline{m}\in\Z^{n-1}$, and $0\leq j\leq r-1.$

\noindent Direct calculation gives us the following:
$$\begin{aligned}
    \mathrm{exp\, ad}(e_0)(f_{\theta_s}^{(j)}\otimes t_1^{rm_1+j}t^{\underline{m}})&=f_{\theta_s}^{(j)}\otimes t_1^{rm_1+j}t^{\underline{m}},   
\\
  \mathrm{exp\, ad}(-f_0)(f_{\theta_s}^{(j)}\otimes t_1^{rm_1+j}t^{\underline{m}})&=f_{\theta_s}^{(j)}\otimes t_1^{rm_1+j}t^{\underline{m}}-[e^{(r-1)}_{\theta_s},f_{\theta_s}^{(j)}]\otimes t_1^{rm_1+j-1}t^{\underline{m}}\\&\,\,\,\,\,\,\,\,+ (e^{(r-1)}_{\theta_s}|f_{\theta_s}^{(j)})t_1^{rm_1+j-1}t^{\underline{m}}K_1 -e^{(r-j)}_{\theta_s}\otimes t_1^{rm_1+j-2}t^{\underline{m}},\end{aligned}$$
  $$\begin{aligned}
 \mathrm{exp\, ad}(e_0)([e^{(r-1)}_{\theta_s},f_{\theta_s}^{(j)}]\otimes t_1^{rm_1+j-1}t^{\underline{m}})&=[e^{(r-1)}_{\theta_s},f_{\theta_s}^{(j)}]\otimes t_1^{rm_1+j-1}t^{\underline{m}}+2f_{\theta_s}^{(j)}\otimes t_1^{rm_1+j}t^{\underline{m}},\\
    \mathrm{exp\, ad}(e_0)(e^{(r-j)}_{\theta_s}\otimes t_1^{rm_1+j-2}t^{\underline{m}})&=e^{(r-j)}_{\theta_s}\otimes t_1^{rm_1+j-2}t^{\underline{m}}+[f^{(1)}_{\theta_s},e^{(r-j)}_{\theta_s}]\otimes t_1^{rm_1+j-1}t^{\underline{m}}\\&\,\,\,\,\,\,\,\,+(f^{(1)}_{\theta_s}|e^{(r-j)}_{\theta_s})t_1^{rm_1+j-1}t^{\underline{m}}K_1-f_{\theta_s}^{(j)}\otimes t_1^{rm_1+j}t^{\underline{m}},\\
[e^{(r-1)}_{\theta_s},f_{\theta_s}^{(j)}]&=-[f^{(1)}_{\theta_s},e^{(r-j)}_{\theta_s}] \,\text{for all}\, 0\leq j\leq r-1.
\end{aligned}$$
    \noindent Hence,
   $$ \begin{aligned}
        \Psi_0(f_{\theta_s}^{(j)}\otimes t_1^{rm_1+j}t^{\underline{m}})&=f_{\theta_s}^{(j)}\otimes t_1^{rm_1+j}t^{\underline{m}}-[e^{(r-1)}_{\theta_s},f_{\theta_s}^{(j)}]\otimes t_1^{rm_1+j-1}t^{\underline{m}}\\&\,\,\,\,\,\,\,\,-2f_{\theta_s}^{(j)}\otimes t_1^{rm_1+j}t^{\underline{m}}+(e^{(r-1)}_{\theta_s}|f_{\theta_s}^{(j)})t_1^{rm_1+j-1}t^{\underline{m}}K_1\\&\,\,\,\,\,\,\,\,-e^{(r-j)}_{\theta_s}\otimes t_1^{rm_1+j-2}t^{\underline{m}}-[f^{(1)}_{\theta_s},e^{(r-j)}_{\theta_s}]\otimes t_1^{rm_1+j-1}t^{\underline{m}}
        \\&\,\,\,\,\,\,\,\,-(f^{(1)}_{\theta_s}|e^{(r-j)}_{\theta_s}) t_1^{rm_1+j-1}t^{\underline{m}}K_1+f_{\theta_s}^{(j)}\otimes t_1^{rm_1+j}t^{\underline{m}}\\&=-e_{\theta_s}^{(r-j)}\otimes t_1^{rm_1+j-2}t^{\underline{m}}.
   \end{aligned} $$
\end{proof}
\subsection{Garland identities}\label{subsec4.2}
Now we recall the Garland identities in Lie algebra $U(L(\mathfrak{sl}_2))$. Let $x, y, h$ be a $\mathfrak{sl}_2$ triplet. For indeterminate $u$, consider the power series ${\bf{P}}_h (u)$ defined by
$${\bf{P}}_h (u) = \mathrm{exp} \Big( - \sum_{j = 1}^{\infty}\frac{h \otimes t^{j}}{j} u^{j}\Big) = \sum_{s = 0}^{\infty} {p_{h}^{(s)}}u^s    ,$$
where $p_{h}^{(s)}$ denotes the coefficient of $u^{s}$ in the power series ${\bf{P}}_h (u)$.

The following are the Garland identities:
For $j \geq 1$ we have
$$ (x \otimes t)^{j}(y \otimes 1)^{j+1} - \sum_{m = 0}^{j}
(y \otimes t^{j-m}) p_{h}^{(m)} \in U(L(\mathfrak{sl}_2)) (L(\mathbb{C} x)),$$
$$(x \otimes t)^{j+1}(y \otimes 1)^{j+1} - p_{h}^{(j+1)}
\in U(L(\mathfrak{sl}_2)) (L(\mathbb{C} x)).$$
\section{Global and Local Weyl Modules} \label{secgl}
This section contains the definitions of global and local Weyl modules along with their properties. We begin with:
\subsection{Character}
Let $\mathcal{V}$ be a $\mathcal{T}(\mu)$-module (hence $\mathcal{T}^+(\mu)$,\,\,$\mathcal{T}_{\mathrm{aff}}(\mu)$-module); furthermore, as a $\mathcal{T}_{\mathrm{aff}}(\mu)$-module, we also assume that $\mathcal{V}$ has the weight space decomposition and all the weight spaces are finite dimensional. For a parameter $q_1$ we define the $q_1$-character of the $\mathcal{T}_{\mathrm{aff}}(\mu)$-module $\mathcal{V}$ by
 $$\mathrm{ch_{q_1}}(\mathcal{V})\coloneqq\sum_{\lambda\in \mathfrak{h}_0^*, m\in \Z}(\mathrm{dim}\mathcal{V}_{\lambda-m\delta_1})e^{\lambda}q_1^m.$$
 Note that each weight space $\mathcal{V}_{\lambda-m\delta_1}$ has the following $\Z^{n-1}$-graded decomposition:
 $$\mathcal{V}_{\lambda-m\delta_1}=\sum_{\underline{p} \in {\Z}^{n-1}}\mathcal{V}_{\lambda-m\delta_1}[\underline{p}].$$
 For parameters $q_1,q_2,\dots,q_n$ we define the $(q_1,q_2,\dots,q_n)$-character of the $\mathcal{T}^+(\mu)$-module $\mathcal{V}$ by
 $$\mathrm{ch_{q_1,q_2,\dots,q_n}}(\mathcal{V})\coloneqq\underset{\underline{p}\in{\Z}^{n-1}}{\sum_{\lambda\in \mathfrak{h}_0^*,m\in\Z}}(\mathrm{dim}\mathcal{V}_{\lambda-m\delta_1}[\underline{p}])e^{\lambda}q_1^mq_2^{p_2}\cdots q_n^{p_n}.$$
 Let $\mathcal{V}_1$ and $\mathcal{V}_2$ be two $\mathcal{T}(\mu)$-modules such that the $q_1$-characters of $\mathcal{V}_1$ and $\mathcal{V}_2$ are well defined. We say that $\mathrm{ch_{q_1}}(\mathcal{V}_1)\leq \mathrm{ch_{q_1}}(\mathcal{V}_2)$ if  $\mathrm{dim}{(\mathcal{V}_1})_{\lambda-m\delta_1}\leq \mathrm{dim}{(\mathcal{V}_2})_{\lambda-m\delta_1}$ for all $\lambda\in\mathfrak{h}_0$ and $m\in\Z$. We can define inequality for $(q_1,q_2,\dots,q_n)$-characters similarly.
\subsection{Definition of global Weyl modules}
\begin{definition}
Let $\Lambda$ be a dominant integral weight of $\mathcal{T}_{aff}(\mu)$. Then 
\begin{enumerate}[label= \((\arabic*)\)]
\item The global Weyl module $W_{\mathrm{glob}}^\mu(\Lambda)$ is a highest weight module for $\mathcal{T}(\mu)$ with highest weight $\Lambda$ and is generated by $v_\Lambda$ with the following defining relations:
$$ e_{i,\underline m}.v_\Lambda=0 \,  \forall\, i\in \tilde{I}_0,\, \underline m\in \Z^{n-1}, \,h.v_\Lambda=\Lambda(h)v_\Lambda \, \forall\, h\in \mathcal{T}_{aff}(\mu)^0
,$$
$$f_i^{\Lambda(h^{(0)}_i)+1}v_\Lambda=0 \,  \forall\, i\in I_0,\, f_0^{\Lambda(h_{0})+1}v_\Lambda=0,\, t^{\underline m} K_i.v_\Lambda=0 \, \forall\, \underline m \in \Z^{n-1}, \, i=2,3,\dots,n,$$
and $$d_i.v_\Lambda =0 \, \forall \,i=2,3,\dots,n.$$
\item The global Weyl module $W_{\mathrm{glob}}^{\mu}(\Lambda)^+$ is a highest weight module for $\mathcal{T}^+(\mu)$ with the highest weight $\Lambda$ and is generated by $v^+_\Lambda$ with the following defining relations:
$$ e_{i,\underline m}.v_\Lambda=0 \,  \forall\, i\in \tilde{I}_0, \,\underline m\in \Z^{n-1}_{\geq 0},\, h.v_\Lambda=\Lambda(h)v_\Lambda \, \forall\, h\in \mathcal{T}_{aff}(\mu)^0
,$$
$$f_i^{\Lambda(h^{(0)}_i)+1}v_\Lambda=0 \,  \forall\, i\in I_0,  \,f_0^{\Lambda(h_0)+1}v_\Lambda=0,\,\text{and} \, t^{\underline m} K_i.v_\Lambda=0 \, \forall\, \underline m \in \Z^{n-1}_{>0},\,i=2,3\dots,n.
$$\end{enumerate}
\end{definition}
\subsection{\texorpdfstring{The algebra $A^\mu(\Lambda)$}{}} Let $\Lambda$ be a dominant integral weight of $\mathcal{T}_{\mathrm{aff}}(\mu)$. The global Weyl module 
$W_{\mathrm{glob}}^{\mu}(\Lambda)$ can be given a right $\overline{\mathfrak{h}}_{\mathcal{T}(\mu)}$-module as follows: $(Z.v_{\Lambda}).Y = Z .Y. v_{\Lambda}$, $Z \in\overline{\mathcal{T}}(\mu), Y \in 
\overline{\mathfrak{h}}_{\mathcal{T}(\mu)}$.
We denote the annihilator of $v_\Lambda$ in $U(\overline{\mathfrak{h}}_{\mathcal{T}(\mu)})$ by  Ann$_{U(\overline{\mathfrak{h}}_{\mathcal{T}(\mu)})}v_\Lambda$. It is easy to see that Ann$_{U(\overline{\mathfrak{h}}_{\mathcal{T}(\mu)})}v_\Lambda$ is an ideal of $U(\overline{\mathfrak{h}}_{\mathcal{T}(\mu)}).$ Let us denote
$$A^\mu(\Lambda)\coloneqq \frac{U(\overline{\mathfrak{h}}_{\mathcal{T}(\mu)})}{\mathrm{Ann}_{U(\overline{\mathfrak{h}}_{\mathcal{T}(\mu)})}v_\Lambda}$$.

Let $\Lambda = \sum_{i = 0}^{\ell} a_i \Lambda_i + a \delta_1$, where
$a_i \in \mathbb{Z}_{\geq 0}$ and $a \in \mathbb{C}$. Set
$M = \mathbb{C}[t_2^{\pm}, \ldots, t_n^{\pm}]$ and let $P = \sum_{i = 0}^{\ell}a_i$. Let $\mathbb{S}_n$ be the permutation group on
$n$-letters and let $\mathbb{S}_{a_{\Lambda}} = \mathbb{S}_{a_0} \times \cdots \times \mathbb{S}_{a_\ell}$. Define $$B^{\mu}(\Lambda) \,\,(\coloneqq
(M^{\otimes P})^{\mathbb{S}_{a_{\Lambda}}})\coloneqq (M^{\otimes a_0})^{ \mathbb{S}_{a_0}} \otimes \cdots \otimes (M^{\otimes a_\ell})^{\mathbb{S}_{a_l}}.$$ Now for any $b\in M$ we set $$\mathrm{sym}_\Lambda^i(b)=1^{\otimes(a_0+\cdots+a_{i-1})}\otimes (\sum_{k = 0}^{a_i-1}{1^{\otimes k}\otimes b\otimes 1^{a_i-k-1}})\otimes 1^{\otimes(a_{i+1}+\cdots+a_{\ell})}.$$ We have the following useful lemma:
\begin{lemma}\label{Lemma1}
   As an algebra over $\C$ the set $$\bigcup_{i\in\tilde{I}_0}\{\mathrm{sym}_\Lambda^i(t_j^q):\,\, q\in \Z,\,\, |q|\leq a_i,\,\,2\leq j\leq n\}$$ generates $B^{\mu}(\Lambda)$.
\end{lemma}
\begin{proof}
     The proof follows by a similar argument as that of the proof of Lemma 5.1 of \cite{FMS}.
\end{proof}
Let $\alpha\in R_0$ and $m\in \Z$ such that $\beta=\alpha+m\delta_1\in R_{\mathrm{aff}}^{\mu}(+)$. Consider the $\mathfrak{sl}_2$ triplets $\{x^{(m)}_{\alpha} \otimes t_1^{m}, x^{(r-m)}_{-\alpha} \otimes t_1^{-m}, \beta_0^{\vee}\} $ and $\{x^{(0)}_{\alpha}, x^{(0)}_{-\alpha}, \alpha^{\vee}\}$, where $\beta_0^{\vee} = [x^{(m)}_{\alpha} \otimes t_1^{m},x^{(r-m)}_{-\alpha} \otimes t_1^{-m}]$ which we denote by $\mathfrak{S}$ and $\mathfrak{sl}_2$, respectively. For any monomial $b\in M$ we have a Lie algebra isomorphism $\psi_{\alpha^\vee, b}^{(m)}: \mathfrak{sl}_2 \otimes \mathbb{C}[z^{\pm1}] \rightarrow \mathfrak{S} \otimes \mathbb{C}[b^{\pm1}]$ given by
$$x^{(0)}_{\alpha} \otimes z^{k} \mapsto x^{(m)}_{\alpha} \otimes t_1^{ m}\otimes b^{k} ,\,\,x^{(0)}_{-\alpha} \otimes z^{k} \mapsto x^{(r-m)}_{-\alpha} \otimes t_1^{-m}\otimes b^{k} ,\,\, \alpha^{\vee} \otimes z^{k} \mapsto \beta_0^{\vee} \otimes b^{k}.$$ Using this isomorphism and the Garland identities (Subsection \ref{subsec4.2}) , we have the following:
\begin{equation}\label{eq10}
(x^{(m)}_{\alpha} \otimes t_1^{m} \otimes b)^{j}(x^{(r-m)}_{-\alpha} \otimes t_1^{-m} \otimes 1)^{j+1}v_{\Lambda} = \sum_{k = 0}^{j}{x^{(r-m)}_{-\alpha} \otimes t_1^{-m} \otimes b^{j- k} p_{\beta_0^{\vee},b}^{(k)}v_{\Lambda}},\end{equation}
and 
\begin{equation}\label{eq11}
(x^{(m)}_{\alpha} \otimes t_1^{m} \otimes b)^{j+1}(x^{(r-m)}_{-\alpha} \otimes t_1^{-m} \otimes 1)^{j+1}v_{\Lambda} = p_{\beta_0^{\vee},b}^{(j+1)}v_{\Lambda},\end{equation}
where $p_{\beta^{\vee}_0,b}^{(k)} = \psi_{\alpha^\vee, b}^{(m)}(p_{\alpha^{\vee}}^{(k)}).$ Now by applying $x^{(m)}_{\alpha} \otimes t_1^{m}$ on both sides of Equation (\ref{eq10}), we get that
\begin{equation}\label{eq12}
(x^{(m)}_{\alpha} \otimes t_1^{m})(x^{(m)}_{\alpha} \otimes t_1^{m} \otimes b)^{j}(x^{(r-m)}_{-\alpha} \otimes t_1^{-m} \otimes 1)^{j+1}v_{\Lambda} = \sum_{k = 0}^{j}{\beta_0^\vee\otimes b^{j- k} p_{\beta_0^{\vee},b}^{(k)}v_{\Lambda}}. \end{equation}
Let $J_\Lambda$ be an ideal of  $U(\overline{\mathfrak{h}}_{\mathcal{T}(\mu)})$ generated by the following elements:
$$h-\Lambda(h),\,\, t^{\underline{s}}K_q,\,\,p_{\alpha_i^{\vee},b}^{(j)}\,\, \text{and}\,\,\sum_{k = 0}^{j-1}{\alpha_i^\vee\otimes b^{j- k} p_{\alpha_i^{\vee},b}^{(k)}v_{\Lambda}},$$
for all $h\in\mathcal{T}_{\mathrm{aff}}(\mu)^0,\,\,\underline{s}\in\Z^{n-1},\,\,2\leq q\leq n,\,\,b\in M,\,\,|j|>a_i\,\,\text{and}\,\,i\in\tilde{I}_0.$
Now, by using Equations (\ref{eq10}), (\ref{eq11}), (\ref{eq12}) and the PBW theorem, one can check that $J_\Lambda$ = Ann$_{U(\overline{\mathfrak{h}}_{\mathcal{T}(\mu)})}v_\Lambda$.
We have the following:
\begin{thm}\label{thm1}
As algebras, $A^{\mu}(\Lambda)$ and $B^\mu(\Lambda)$ are isomorphic.
\end{thm}
\begin{proof}
 Define a map $\psi:$ $U(\overline{\mathfrak{h}}_{\mathcal{T}(\mu)})\rightarrow B^\mu(\Lambda)$ by
 $$\psi(\alpha_i^\vee\otimes b)=\mathrm{sym}_\Lambda^i\big(b\big),\,\,0\leq i\leq \ell,\,\,b\in M$$
 and $\psi(1)=1^{\otimes P}$. It is easy to see that $\psi$ is an algebra homomorphism and $\psi(t^{\underline{s}}K_q)=0,\,\,2\leq q\leq n.$ In fact, the Lemma \ref{Lemma1} proves that the map $\psi$ is a surjective algebra homomorphism. An argument similar to that of the proof of Proposition 4.4 of \cite{SSS} proves that the map $\psi$ factors through $J_\Lambda$. Now we have a surjective algebra homomorphism $\tilde {\psi}:  A^{\mu}(\Lambda)\rightarrow B^\mu(\Lambda)$ induced from $\psi$. Also, the injectivity of $\tilde{\psi}$ follows with a similar argument as that of the proof of Lemma 10 of \cite{VGT}.
\end{proof}

Theorem \ref{thm1} implies that $A^{\mu}(\Lambda)$ is finitely generated. Now we have the similar result for $W_{\mathrm{glob}}^{\mu}(\Lambda)^+$ over $\mathcal{T}(\mu)$  as of that of $W_{\mathrm{glob}}^{\mu}(\Lambda)$ over $\mathcal{T}^+(\mu)$ if we just replace $A^{\mu}(\Lambda)$ with $A^{\mu}(\Lambda)^+\coloneqq \frac{U(\overline{\mathfrak{h}}_{\mathcal{T}(\mu)^+})}{\mathrm{Ann}_{U(\overline{\mathfrak{h}}_{\mathcal{T}(\mu)^+})}v_\Lambda}$ which proves that $A^{\mu}(\Lambda)^+$ is finitely generated and isomorphic to $B^{\mu}(\Lambda)^+$ as an algebra over $\C$, where
$$\overline{\mathfrak{h}}_{\mathcal{T}(\mu)^+}=\mathfrak{h}_0\otimes\C[t_2,\dots,t_n]\oplus \C[t_2,\dots,t_n]K_1\oplus {\sum_{i=2}^{n}}t_i\C[t_2,\dots,t_n]K_i\oplus \C d_1,$$
and 
$$B^{\mu}(\Lambda)^+\,\,(\coloneqq
((M^+)^{\otimes P})^{\mathbb{S}_{a_{\Lambda}}})\coloneqq ((M^+)^{\otimes a_0})^{ \mathbb{S}_{a_0}} \otimes \cdots \otimes ((M^+)^{\otimes a_\ell})^{\mathbb{S}_{a_l}},$$
where $M^+ = \mathbb{C}[t_2, \ldots, t_n]$.
\subsection{Definition of local Weyl modules}
For maximal ideals $\mathcal{M}$ and $\mathcal{M}^+$  of $A^\mu({\Lambda})$ and $A^\mu({\Lambda})^+$, respectively, let $\C_{\mathcal{M}}$ and $\C_{\mathcal{M}^+}$ represent the corresponding one dimensional modules $A^\mu({\Lambda})/\mathcal{M}$ and $A^\mu({\Lambda})^+/\mathcal{M}^+$, respectively. The local weyl modules associated with $\C_{\mathcal{M}}$ and $\C_{\mathcal{M}}^+$ are denoted by $W_{\mathrm{loc}}^{\mu}(\Lambda,\mathcal{M})$ and $W_{\mathrm{loc}}^{\mu}(\Lambda,\mathcal{M}^+)^+$, respectively, and they are defined by 
$$W_{\mathrm{loc}}^{\mu}(\Lambda,\mathcal{M})\coloneqq W_{\mathrm{glob}}^{\mu}(\Lambda)\otimes_{A^\mu({\Lambda})}\C_{\mathcal{M}},\,\,W_{\mathrm{loc}}^{\mu}(\Lambda,\mathcal{M}^+)^+\coloneqq W_{\mathrm{glob}}^{\mu}(\Lambda)^+\otimes_{A^\mu({\Lambda})^+}\C_{\mathcal{M}^+}.$$
In the local Weyl modules, we denote the images of $v_\Lambda$ and $v_\Lambda^+$ by $v_{\Lambda,\mathcal{M}}^+$ and $v_{\Lambda,\mathcal{M}^+}^+$, respectively.
\begin{lemma} \label{prom 4} Let $\Lambda$ be a dominant integral weight of $\mathcal{T}_{\mathrm{aff}}(\mu)$. Then we have:
\begin{enumerate}[label= \((\arabic*)\)]
\item For any root $\beta \in R_{\mathrm{aff}}^{\mu}(+)$, there exists $N(\beta) \in \Z_{\geq 0}$ such that for any $\underline q\in \Z^{n-1}$ we have
$$
    x_{-\beta}\otimes t^{\underline q}.v_\Lambda \in \underset{2 \leq i \leq n } {\sum_{p_i=0}^{N(\beta)}} x_{-\beta}\otimes t^{\underline p}A^\mu(\Lambda).v_\Lambda
.$$
\item For any $s\in r\Z_{>0}$ there exists $N(s) \in \Z_{\geq 0} $ such that for any $\underline q\in \Z^{n-1}$ and $j=2,3,\dots,n$ we have 
$$   t_1^{-s} t^{\underline q}K_j .v_\Lambda \in \underset{2 \leq i \leq n } {\sum_{p_i=0}^{N(s)}} t_1^{-s}t^{\underline p}K_j A^\mu(\Lambda).v_\Lambda + \underset{2 \leq i \leq n } {\sum_{p_i=0}^{N(s)}} (\mathcal{T}_{aff}(\mu))_{-s\delta_1}\otimes t^{\underline p}A^\mu(\Lambda).v_\Lambda 
.$$\end{enumerate}
\end{lemma}
\begin{proof}
 Let $\alpha \in R_0$ such that $\beta = \alpha + m\delta_1 \in R_{\mathrm{aff}}^{\mu}(+)$, for appropriate $m \in \mathbb{Z}_{> 0}$. Let $b = t_3^{b_3}t_4^{b_4} \cdots t_n^{b_n}$ be a monomial in $A_n$. Consider the sets $\mathfrak{S} \coloneqq \{x^{(m)}_{\alpha} \otimes t_1^{m}b^{-1}, x^{(r-m)}_{-\alpha} \otimes t_1^{-m}b, \beta_0^{\vee}\} $ and $\{x^{(0)}_{\alpha}, x^{(0)}_{-\alpha}, \alpha^{\vee}\}$ which form $\mathfrak{sl}_2$
triplets, where $\beta_0^{\vee} = [x^{(m)}_{\alpha} \otimes t_1^{m}b^{-1},x^{(r-m)}_{-\alpha} \otimes t_1^{-m}b]$. We have the Lie algebra isomorphism $\psi_{2, b}: \mathfrak{sl}_2 \otimes \mathbb{C}[t_2^{\pm1}] \rightarrow \mathfrak{S} \otimes \mathbb{C}[t_2^{\pm1}]$ given by
$$x^{(0)}_{\alpha} \otimes t_2^{k} \mapsto x^{(m)}_{\alpha} \otimes t_1^{m}b^{-1} \otimes t_2^{k},\,\, x^{(0)}_{- \alpha} \otimes t_2^{k} \mapsto x^{(r-m)}_{-\alpha} \otimes t_1^{m}b \otimes t_2^{k},\,\, \alpha^{\vee} \otimes t_2^{k}\mapsto \beta_0^{\vee} \otimes t_2^{k}.$$
 Then, by an application of Garland identity we have
$$(x^{(m)}_{\alpha} \otimes t_1^{m}b^{-1} \otimes t_2)^{s}(x^{(r-m)}_{-\alpha} \otimes t_1^{-m}b \otimes 1)^{s+1}v_{\Lambda} = \sum_{k_2 = 0}^{s}{x^{(r-m)}_{-\alpha} \otimes t_1^{-m}b \otimes t_2^{s - k_2} p_{\beta_0^{\vee}}^{(k_2)}v_{\Lambda}},$$
where $p_{\beta_0^{\vee}}^{(k_2)} = \psi_{2, b}(p_{\alpha^{\vee}}^{(k_2)}).$ Hence we have
$$\sum_{k_2 = 0}^{\Lambda(\alpha^{\vee})}{(x^{(r-m)}_{-\alpha} \otimes t_1^{-m}) t_2^{{\Lambda(\alpha^{\vee})} - k_2}t_3^{b_3}t_4^{b_4} \cdots t_n^{b_n} {p_{\beta_0^{\vee}}^{(k_2)}}v_{\Lambda}} = 0.$$
Now, using a similar argument as that of the proof of Lemma 4.5 of \cite{SSS}, we get the desired result for this case.
The arguments for $\beta = m \delta_1$ and $\beta = 2 \alpha + (2m+1) \delta$ where $\alpha \in R^s_0$ and $m \in \mathbb{Z}_{>0}$ also can be dealt similarly as that of proof Lemma 4.5 of \cite{SSS}.

\noindent
The proof of (2) can be achieved using the similar argument as that of proof of Lemma 3.10 (ii) of \cite{KR} 
\end{proof}
\begin{prop} \label{prop2}
Let $\Lambda$ be a dominant integral weight of $\mathcal{T}_{\mathrm{aff}}(\mu)$.  For any root $\beta \in R_{\mathrm{aff}}^{\mu}(+)$ \begin{enumerate}[label= \((\arabic*)\)]
\item The weight spaces ${W_{\mathrm{glob}}^{\mu}}(\Lambda)_{\Lambda-\beta}$ and ${W_{\mathrm{glob}}^{\mu}}(\Lambda)^+_{\Lambda-\beta} $ are finitely generated over $A^{\mu}(\Lambda)$ and $A^{\mu}(\Lambda)^+$, respectively.
\item The weight spaces $W_{\mathrm{loc}}^\mu(\Lambda,\mathcal{M})_{\Lambda-\beta}$ and $W_{\mathrm{loc}}^\mu(\Lambda,\mathcal{M}^+)_{\Lambda-\beta}^+$ are finite dimensional.
\item  We have $W_{\mathrm{loc}}^\mu(\Lambda,\mathcal{M})=U(\mathcal{T}^+(\mu))v_{\Lambda,\mathcal{M}}$.
\end{enumerate}
\end{prop}
\begin{proof}
    The proof follows using Lemma \ref{prom 4}.
\end{proof}
\section{Global and Local Weyl Modules of Level One} \label{secl1}
In this section, we will concentrate on the dominant integral weights of $\mathcal{T}_{\mathrm{aff}}(\mu)$ of level one. It is enough to consider the fundamental weight $\Lambda_0$ (see \cite{KAC} \S 12.4 ). By Theorem
\ref{thm1} we have $A^{\mu}(\Lambda_0) = \mathbb{C}[y_2^{\pm 1}, \ldots, y_n^{\pm 1}]$ and $A^{\mu}(\Lambda_0)^+ = \mathbb{C}[y_2, \ldots, y_n]$.
It is well known that the maximal ideals of $A^{\mu}(\Lambda_0)$ and $A^{\mu}(\Lambda_0)^+$ 
are in one-to-one correspondence with elements of $({\mathbb{C}^*})^{n-1}$ and $\mathbb{C}^{n-1}$.
Let $\underline{a}$ and $\underline{b}$ be the elements of ${(\mathbb{C}^*})^{n-1}$ and $\mathbb{C}^{n-1}$, respectively, and let $\mathcal{M}$ and $\mathcal{M}^+$ be the corresponding maximal ideals of 
$A^{\mu}(\Lambda_0)$ and $A^{\mu}(\Lambda_0)^+$, respectively.

For the rest of this paper, we will write $W_{\mathrm{loc}}^{\mu}(\Lambda_0, \underline{a}), W_{\mathrm{loc}}^{\mu}(\Lambda_0, \underline{b})^{+} $, $v_{\Lambda_0,\underline{a}}$, $v_{\Lambda_0,\underline{b}}^+$ for $W_{\mathrm{loc}}^{\mu}(\Lambda_0, \mathcal{M}), W_{\mathrm{loc}}^{\mu}(\Lambda_0, \mathcal{M}^{+})^{+}$, $v_{\Lambda_0,\mathcal{M}}$, $v_{\Lambda_0,\mathcal{M}^+}^+$, respectively.
The right action of $A^{\mu}(\Lambda_0)$ on $W_{\mathrm{glob}}^{\mu}(\Lambda_0)$ is given by $ X.v_{\Lambda_0}. y^{\underline{m}} = 
X.({\tilde {\psi}}^{-1}(y^{\underline{m}}).v_{\Lambda_0}) = X(h_{0, \underline{m}}.v_{\Lambda_0})$, where $X\in U(\mathcal{T}(\mu))$, $\underline{m}\in\Z^{n-1}$ and $y^{\underline{m}}=y_2^{m_2}\cdots y_n^{m_n}\in A^{\mu}(\Lambda_0)$. In particular, for $W_{\mathrm{loc}}^{\mu}(\Lambda_0, \underline{a})$, we have $h_{0, \underline{m}}.v_{\Lambda_0,\underline{a}} = \underline{a}^{\underline{m}}v_{\Lambda_0,\underline{a}},$  where $v_{\Lambda_0,\underline{a}}$ denotes the highest weight vector of 
$W_{\mathrm{loc}}^{\mu}(\Lambda_0, \underline{a})$ and $\underline{m} \in \mathbb{Z}^{n-1}$. Similarly for 
$W_{\mathrm{loc}}^{\mu}(\Lambda_0, \underline{b})^+$, we have
$h_{0, \underline{k}}.v^+_{\Lambda_0,\underline{b}} = \underline{b}^{\underline{k}}v^+_{\Lambda_0,\underline{b}}$, where
$v^+_{\Lambda_0,\underline{b}}$ denotes the highest weight vector of 
$W_{\mathrm{loc}}^{\mu}(\Lambda_0, \underline{b})^+$ and $\underline{k} \in \mathbb{Z}^{n-1}_{\geq 0}$. The Proposition \ref{prop2} implies that $W_{\mathrm{loc}}^\mu(\Lambda,\mathcal{M})=U(\mathcal{T}^+(\mu))\otimes_{A^\mu({\Lambda_0})}\C v_{\Lambda_0,\mathcal{M}}$ and similarly $W_{\mathrm{loc}}^\mu(\Lambda,\mathcal{M}^+)^+=U(\mathcal{T}^+(\mu))\otimes_{A^\mu({\Lambda_0})^+}\C v^+_{\Lambda_0,\mathcal{M}^+}$. One can observe that $h_{i,\underline{m}}v_{\Lambda_0}=0$ for all $i\in I_0$ and   $\underline{m}\in\Z^{n-1}.$
\begin{lemma}
    We have $$\mathrm{ch}_{q_1}W_{\mathrm{loc}}^{\mu}(\Lambda_0,\underline{b})^+=\mathrm{ch}_{q_1}W_{\mathrm{loc}}^{\mu}(\Lambda_0,\underline{0})^+$$ for all $\underline{b}\in\C^{n-1}$.
\end{lemma}
\begin{proof}
    The proof follows by a similar argument as that of the proof of Proposition 3.14 of \cite{KR}.
\end{proof}
For the rest of this section, we will be working with $W_{\mathrm{loc}}^{\mu}(\Lambda_0,\underline{0})^+$, and we denote $v_{\Lambda_0,\underline{0}}^+$ by $v_0$. The following relations hold in $W_{\mathrm{loc}}^{\mu}(\Lambda_0,\underline{0})^+$:
$$ e_{i,\underline m}.v_0=0 \,\,  \forall\,i\in \tilde{I}_0,\, \underline m\in \Z^{n-1}_{\geq 0},\,\, h.v_0=\Lambda(h)v_0 \, \forall\, h\in \mathcal{T}_{aff}(\mu)^0,
$$
$$ h_{i,\underline m}.v_0=0, \,\,f_{i,\underline m}.v_0=0 \,\,  \forall\,i\in {I}_0,\, \underline m\in \Z^{n-1}_{\geq 0},
$$
$$f_i.v_0=0 \,\,  \forall\,i\in I_0,\,\, f_0^2.v_0=0,\,\,h_{0,\underline{m}}.v_0=0\,\,\forall\,\underline m\in \Z^{n-1}_{\geq 0}\setminus\{0\},$$
by using  $0=[e_{0,\underline{m}},f_0^2].v_0=(-2f_{0,\underline{m}}+2f_0h_{0,\underline{m}}).v_0$ and $h_{0,\underline{m}}.v_0=0\,\,\forall\underline m\in \Z^{n-1}_{\geq 0}\setminus\{0\}$ we have
$$f_{0,\underline{m}}.v_0=0\,\,\forall\underline m\in \Z^{n-1}_{\geq 0}\setminus\{0\}.$$
\begin{lemma}\label{prom 5}
Let $\underline m=(m_2,m_3,\dots,m_n)\in \Z^{n-1}_{\geq 0}$ be such that $p=m_2+m_3+\dots+m_n\geq 1$ and $l\in \Z$. Then we have
 \begin{enumerate}[label= \((\arabic*)\)]
\item

\begin{equation}
  e^{(j)}_{{\theta}_s}\otimes t_1^{-(rl-j)}t^{\underline m} v_0= \begin{cases}
    0, & \text{if $l\leq p$},\\[.5em]
r\displaystyle \sum_{q=1}^{l-p}(t_1^{r(-l+q)}t^{\underline m}K_1)(e^{(j)}_{{\theta}_s}\otimes t_1^{-(rq-j)})v_0, & \text{if $l> p$},
  \end{cases}
\end{equation}
for $0\leq j \leq r-1.$
\item

\begin{equation}
  t_1^{-rl}t_2^{m_2}\cdots t_i^{m_i+1}\cdots t_n^{m_n}K_iv_0 = \begin{cases}
    0, & \text{if $l\leq p$},\\[.5em]
\displaystyle r\sum_{q=1}^{l-p}(t_1^{r(-l+q)}t^{\underline m}K_1)(t_1^{-rq}t_iK_i)v_0, & \text{if $l> p$},
  \end{cases}
\end{equation} for $i=2,3,\dots,n.$
\end{enumerate}
\end{lemma}
\begin{proof}
   We prove assertions (1) and (2) using mathematical induction on $l$. Assume that $l\leq 0$, then clearly $-(rl-j)\geq 0$, and hence $e^{(j)}_{{\theta}_s}\otimes t_1^{-(rl-j)}t^{\underline m}\in \mathcal{T}^+(\mu)\cap \mathcal{T}(\mu)^{+}$, which imply that $e^{(j)}_{{\theta}_s}\otimes t_1^{-(rl-j)}t^{\underline m}v_0=0.$
   
   \noindent Next, assume that $l=1,$ then clearly for $j=r-1$ we have $e^{(j)}_{{\theta}_s}\otimes t_1^{-(rl-j)}t^{\underline m}=f_0\otimes t^{\underline m}, $ which imply that $e^{(j)}_{{\theta}_s}\otimes t_1^{-(rl-j)}t^{\underline m}.v_0=0.$ For $j=r-2$, by Lemma \ref{RP1} we have
$$\begin{aligned}
e^{(r-2)}_{{\theta}_s}\otimes t_1^{-2}t^{\underline m} v_0&=\Psi_0\Psi_{\theta_s}\Big(e^{(0)}_{{\theta}_s}\otimes t^{\underline m}\Psi^{-1}_{\theta_s}\Psi^{-1}_0(v_0)\Big) \\
&= \Psi_0\Psi_{\theta_s}\Big(e^{(0)}_{{\theta}_s}\otimes t^{\underline m} \Big(-f^{(r-1)}_{\theta_s}\otimes t^{-1}_1\Big)v_0\Big)\\ &{\qquad\qquad(\because\text{$\Psi^{-1}_{\theta_s}\Psi^{-1}_0(v_0)=-f^{(r-1)}_{\theta_s}\otimes t^{-1}_1v_0$})}\\
&=\Psi_0\Psi_{\theta_s}\Big(\Big[f^{(r-1)}_{\theta_s}\otimes t^{-1}_1, e^{(0)}_{{\theta}_s}\otimes t^{\underline m}\Big]v_0\Big)\\
&=\Psi_0\Psi_{\theta_s}\Big(\Big[f^{(0)}_{\theta_s}, e^{(r-1)}_{{\theta}_s}\otimes t^{-1}_1 t^{\underline m}\Big]v_0\Big)=0.
\end{aligned}$$
Following a similar argument as for the case of $j=r-2$, and using $e^{(r-2)}_{{\theta}_s}\otimes t_1^{-2}t^{\underline m}v_0=0$, we can prove that $e^{(j)}_{{\theta}_s}\otimes t_1^{-(r-j)}t^{\underline m}v_0=0\,\,\text{for all}\,\, j= r-1, \dots,0.$

\noindent Now we have,
$$\begin{aligned}
t_1^{-rl}t_2^{m_2}\cdots t_i^{m_i+1}&\cdots t_n^{m_n}K_iv_0\\&= \frac{1}{r}\Big(\Big[f^{(0)}_{\theta_s}\otimes t_i, e^{(0)}_{\theta_s}\otimes t_1^{-rl}t^{\underline m}\Big]- \Big[f^{(0)}_{\theta_s}, e^{(0)}_{\theta_s}\otimes t_1^{-rl}t_2^{m_2}\cdots t_i^{m_i+1}\cdots t_n^{m_n}\Big]\Big)v_0\\
&=0, \,\,\text{for $l\leq 1$}.
\end{aligned}$$
Thus, we proved the assertions (1) and (2) for $l\leq 1$.

\noindent Let $l>1$ and assume that the assertions (1) and (2) are true for all $l'<l$. For $j=r-1$ we have
$$\begin{aligned}
  e^{(r-1)}_{{\theta}_s}\otimes t_1^{-(rl-(r-1))}t^{\underline m}  &=\Psi_0\Psi_{\theta_s}\Big(e^{(1)}_{{\theta}_s}\otimes t^{-(r(l-1)-1)}_1 t^{\underline m}\Psi^{-1}_{\theta_s}\Psi^{-1}_0(v_0)\Big)\\
&=\Psi_0\Psi_{\theta_s}\Big(e^{(1)}_{{\theta}_s}\otimes t^{-(r(l-1)-1)}_1t^{\underline m}\Big(-f^{(r-1)}_{\theta_s}\otimes t^{-1}_1\Big)v_0\Big)\\
&=\Psi_0\Psi_{\theta_s}\Big(-f^{(r-1)}_{\theta_s}\otimes t^{-1}_1e^{(1)}_{{\theta}_s}\otimes t^{-(r(l-1)-1)}_1t^{\underline m}v_0\\& \,\,\,\,\,\,\,\,+\Big[f^{(r-1)}_{\theta_s}\otimes t^{-1}_1, e^{(1)}_{{\theta}_s}\otimes t^{-(r(l-1)-1)}_1t^{\underline m} \Big]v_0\Big)\\
&=\Psi_0\Psi_{\theta_s}\Big(-f^{(r-1)}_{\theta_s}\otimes t^{-1}_1e^{(1)}_{{\theta}_s}\otimes t^{-(r(l-1)-1)}_1t^{\underline m}v_0\\
&\,\,\,\,\,\,\,\,+\Big[f^{(r-1)}_{\theta_s}, e^{(1)}_{{\theta}_s} \Big] \otimes t^{-r(l-1)}_1t^{\underline m}v_0-(f^{(r-1)}_{\theta_s}|e^{(1)}_{{\theta}_s})t^{-r(l-1)}_1t^{\underline m}K_1v_0\Big)\\&=\Psi_0\Psi_{\theta_s}\Big(-f^{(r-1)}_{\theta_s}\otimes t^{-1}_1e^{(1)}_{{\theta}_s}\otimes t^{-(r(l-1)-1)}_1t^{\underline m}v_0 \\
&\,\,\,\,\,\,\,\,+ \Big[f^{(0)}_{\theta_s}, e^{(0)}_{{\theta}_s} \Big] \otimes t^{-r(l-1)}_1t^{\underline m}v_0 - rt^{-r(l-1)}_1t^{\underline m}K_1v_0\Big)
\end{aligned}$$
\begin{align}\label{eq3}
\begin{split}
e^{(r-1)}_{{\theta}_s}\otimes t_1^{-(rl-(r-1))}t^{\underline m} &=\Psi_0\Psi_{\theta_s}\Big(-f^{(r-1)}_{\theta_s}\otimes t^{-1}_1e^{(1)}_{{\theta}_s}\otimes t^{-(r(l-1)-1)}_1t^{\underline m}v_0 \\
&\,\,\,\,\,\,\,\,+ f^{(0)}_{\theta_s} \Big(e^{(0)}_{{\theta}_s}\otimes t^{-r(l-1)}_1t^{\underline m}\Big)v_0 - rt^{-r(l-1)}_1t^{\underline m}K_1v_0\Big).
\end{split}
\end{align}
By the assumption of mathematical induction, we have
\begin{align}\label{eq4}
  e^{(1)}_{{\theta}_s}\otimes t^{-(r(l-1)-1)}_1t^{\underline m}v_0&=\begin{cases} 0, & \text{if}\,\, l-1\leq p,\\[.5em]
  r\displaystyle \sum_{q=1}^{l-p-1}(t_1^{r(-l+q+1)}t^{\underline m}K_1)(e^{(1)}_{{\theta}_s}\otimes t_1^{-(rq-1)})v_0, & \text{if}\,\, l-1> p,
      \end{cases} 
\end{align}
and 
\begin{align}\label{eq5}
  e^{(0)}_{{\theta}_s}\otimes t^{-r(l-1)}_1t^{\underline m}v_0&=\begin{cases} 0, & \text{if}\,\, l-1\leq p,\\[.5em]
  r\displaystyle \sum_{q=1}^{l-p-1}(t_1^{r(-l+q+1)}t^{\underline m}K_1)(e^{(0)}_{{\theta}_s}\otimes t_1^{-rq})v_0, & \text{if}\,\, l-1> p.
      \end{cases} 
\end{align}
Using Equations (\ref{eq4}) and (\ref{eq5}) the right-hand side of Equation (\ref{eq3}) is equal to
    $$\begin{aligned}\Psi_0\Psi_{\theta_s}&\Big(-rt^{-r(l-1)}_1t^{\underline m}K_1v_0\Big)\,\, \text{if}\,\, l-1\leq p, \end{aligned}$$ 
    and to $$\begin{aligned}\Psi_0\Psi_{\theta_s}&\Big(-rf^{(r-1)}_{\theta_s}\otimes t^{-1}_1\Big(\displaystyle \sum_{q=1}^{l-p-1}(t_1^{r(-l+q+1)}t^{\underline m}K_1)(e^{(1)}_{{\theta}_s}\otimes t_1^{-(rq-1)})\Big)v_0
\\&+ rf^{(0)}_{\theta_s} \displaystyle \sum_{q=1}^{l-p-1}(t_1^{r(-l+q+1)}t^{\underline m}K_1)(e^{(0)}_{{\theta}_s}\otimes t_1^{-rq})v_0 - rt^{-r(l-1)}_1t^{\underline m}K_1v_0\Big)\,\,\text{if}\,\, l-1> p.
\end{aligned}$$
\noindent{\bf{Case (1) :}} If $l-1<p$, then by using assertion (2), we have
$$\begin{aligned}
    e^{(r-1)}_{{\theta}_s}\otimes t_1^{-(rl-(r-1))}t^{\underline m}&=\Psi_0\Psi_{\theta_s}\Big(-rt^{-r(l-1)}_1t^{\underline m}K_1v_0\Big)\\
    &=\Psi_0\Psi_{\theta_s}\Big(\frac{r}{-r(l-1)}\sum_{i=1}^{n}m_it^{-r(l-1)}_1t^{\underline m}K_iv_0 \Big)=0.
\end{aligned}$$
\noindent{\bf{Case (2) :}} If $l-1=p$, then by using \text{$\Psi_0\Psi_{\theta_s}(v_0)=-e^{(r-1)}_{\theta_s}\otimes t^{-1}_1v_0$} we have
$$\begin{aligned}
    e^{(r-1)}_{{\theta}_s}\otimes t_1^{-(rl-(r-1))}t^{\underline m}&=\Psi_0\Psi_{\theta_s}\Big(-rt^{-r(l-1)}_1t^{\underline m}K_1v_0\Big)\\
    &=-rt^{-r(l-1)}_1t^{\underline m}K_1 \Psi_0\Psi_{\theta_s}(v_0)\\
    &=rt^{-r(l-1)}_1t^{\underline m}K_1e^{(r-1)}_{\theta_s}\otimes t^{-1}_1v_0.
\end{aligned}$$
\noindent{\bf{Case (3) :}} If $l-1>p$, then by using \text{$\Psi_0\Psi_{\theta_s}(v_0)=-e^{(r-1)}_{\theta_s}\otimes t^{-1}_1v_0$} we have
\begin{align}\label{eq6}
e^{(r-1)}_{{\theta}_s}\otimes t_1^{-(rl-(r-1))}t^{\underline m}&=r\sum_{q=1}^{l-p-1}t_1^{r(-l+q+1)}t^{\underline m}K_1\Psi_0\Psi_{\theta_s}\Big(-f^{(r-1)}_{\theta_s}\otimes t^{-1}_1(e^{(1)}_{{\theta}_s}\otimes t_1^{-(rq-1)})v_0
\notag\\&\,\,\,\,\,\,\,\,+ f^{(0)}_{\theta_s} (e^{(0)}_{{\theta}_s}\otimes t_1^{-rq})v_0\Big)+ rt^{-r(l-1)}_1t^{\underline m}K_1e^{(r-1)}_{\theta_s}\otimes t^{-1}_1v_0.
\end{align}
\noindent{\bf{Claim 1 :}} We have $$e^{(r-1)}_{\theta_s}\otimes t_1^{-(rq+1)}=\Psi_0\Psi_{\theta_s}\Big(-f^{(r-1)}_{\theta_s}\otimes t^{-1}_1(e^{(1)}_{{\theta}_s}\otimes t_1^{-(rq-1)})v_0+ f^{(0)}_{\theta_s} (e^{(0)}_{{\theta}_s}\otimes t_1^{-rq})v_0\Big).$$

\noindent Proof of Claim 1:
$$\begin{aligned}
    e^{(r-1)}_{\theta_s}&\otimes t_1^{-(rq+1)}=\Psi_0\Psi_{\theta_s}\Big(e^{(1)}_{{\theta}_s}\otimes t^{-(rq-1)}_1 t^{\underline m}\Psi^{-1}_{\theta_s}\Psi^{-1}_0(v_0)\Big)\\
    &=\Psi_0\Psi_{\theta_s}\Big(-f^{(r-1)}_{\theta_s}\otimes t^{-1}_1(e^{(1)}_{{\theta}_s}\otimes t_1^{-(rq-1)})v_0+[f^{(r-1)}_{\theta_s}\otimes t^{-1}_1,e^{(1)}_{{\theta}_s}\otimes t_1^{-(rq-1)}] \Big)\\
    &=\Psi_0\Psi_{\theta_s}\Big(-f^{(r-1)}_{\theta_s}\otimes t^{-1}_1(e^{(1)}_{{\theta}_s}\otimes t_1^{-(rq-1)})v_0+[f^{(r-1)}_{\theta_s},e^{(1)}_{{\theta}_s}]\otimes t^{-rq}_1 \Big)\\
    &=\Psi_0\Psi_{\theta_s}\Big(-f^{(r-1)}_{\theta_s}\otimes t^{-1}_1(e^{(1)}_{{\theta}_s}\otimes t_1^{-(rq-1)})v_0+[f^{(0)}_{\theta_s},e^{(0)}_{{\theta}_s}]\otimes t^{-rq}_1 \Big)\\
    &=\Psi_0\Psi_{\theta_s}\Big(-f^{(r-1)}_{\theta_s}\otimes t^{-1}_1(e^{(1)}_{{\theta}_s}\otimes t_1^{-(rq-1)})v_0+ f^{(0)}_{\theta_s} (e^{(0)}_{{\theta}_s}\otimes t_1^{-rq})v_0\Big).
    \end{aligned}$$
    Using Claim 1 in Equation (\ref{eq6}), we get that
    \begin{align}\label{eq7}
    e^{(r-1)}_{{\theta}_s}\otimes t_1^{-(rl-(r-1))}t^{\underline m}&=r\sum_{q=1}^{l-p-1}t_1^{r(-l+q+1)}t^{\underline m}K_1e^{(r-1)}_{\theta_s}\otimes t_1^{-(rq+1)}\notag\\ &\,\,\,\,\,\,\,\,+ rt^{-r(l-1)}_1t^{\underline m}K_1e^{(r-1)}_{\theta_s}\otimes t^{-1}_1v_0\notag\\
    &=r\displaystyle \sum_{q=1}^{l-p}(t_1^{r(-l+q)}t^{\underline m}K_1)(e^{(j)}_{{\theta}_s}\otimes t_1^{-(rq-j)})v_0.
\end{align}
Following a similar argument as for the case of $j=r-1$, and using Equation (\ref{eq7}), we can prove the assertion (1) for $j=r-2$, and in this way, we can prove the assertion (1) for all $j= r -1,\dots,0.$
This completes the proof of assertion (1).

\noindent By assertion (1), we have
\begin{align}\label{eq8}
t_1^{-rl}t_2^{m_2}&\cdots t_i^{m_i+1}\cdots t_n^{m_n}K_iv_0 \notag\\&= \frac{1}{r}\Big(\Big[f^{(0)}_{\theta_s}\otimes t_i, e^{(0)}_{\theta_s}\otimes t_1^{-rl}t^{\underline m}\Big]- \Big[f^{(0)}_{\theta_s}, e^{(0)}_{\theta_s}\otimes t_1^{-rl}t_2^{m_2}\cdots t_i^{m_i+1}\cdots t_n^{m_n}\Big]\Big)v_0\notag\\
     &= \frac{1}{r}\Big(f^{(0)}_{\theta_s}\otimes t_i( e^{(0)}_{\theta_s}\otimes t_1^{-rl}t^{\underline m})v_0-f^{(0)}_{\theta_s} (e^{(0)}_{\theta_s}\otimes t_1^{-rl}t_2^{m_2}\cdots t_i^{m_i+1}\cdots t_n^{m_n})v_0\Big).
\end{align}
Using assertion (1), the right-hand side of Equation (\ref{eq8}) is equal to $$\begin{cases}
    0,& \text{if}\,\, l< p+1,\\[.5em]
    r(t_1^{r(-l+1)}t^{\underline m}K_1)(t_1^{-r}t_iK_i),&\text{if}\,\,l=p+1,
    \end{cases}$$
    and to 
    $$\begin{aligned}
       &f^{(0)}_{\theta_s}\otimes t_i\Big( \sum_{q=1}^{l-p}(t_1^{r(-l+q)}t^{\underline m}K_1)(e^{(0)}_{{\theta}_s}\otimes t_1^{-rq})\Big)v_0\\
       &\,\,\,\,\,\,\,\,-f^{(0)}_{\theta_s} \Big( \sum_{q=1}^{l-(p+1)}(t_1^{r(-l+q)}t_2^{m_2}\cdots t_i^{m_i+1}\cdots t_n^{m_n}K_1)(e^{(0)}_{{\theta}_s}\otimes t_1^{-rq})\Big)v_0,\,\, \text{if}\,\,l>p+1.
    \end{aligned}$$
Thus, the assertion (2) is true for $l\leq p+1$; therefore, we may assume $l>p+1$. We have
    \begin{align}\label{eq9}
        t_1^{-rl}t_2^{m_2}&\cdots t_i^{m_i+1}\cdots t_n^{m_n}K_iv_0\notag\\&=f^{(0)}_{\theta_s}\Big(\sum_{q=1}^{l-p}(t_1^{r(-l+q)}t^{\underline m}K_1)(e^{(0)}_{{\theta}_s}\otimes t_1^{-rq}t_i)\Big)v_0+r\sum_{q=1}^{l-p}(t_1^{r(-l+q)}t^{\underline m}K_1)(t_1^{-rq}t_iK_i)v_0\notag\\
       &\,\,\,\,\,\,\,\,-f^{(0)}_{\theta_s} \Big( \sum_{q=1}^{l-(p+1)}(t_1^{r(-l+q)}t_2^{m_2}\cdots t_i^{m_i+1}\cdots t_n^{m_n}K_1)(e^{(0)}_{{\theta}_s}\otimes t_1^{-rq})\Big)v_0.
    \end{align}
    \noindent{\bf{Claim 2 :}} We have 
    $$\begin{aligned}
\sum_{q=1}^{l-p}(t_1^{r(-l+q)}t^{\underline m}K_1)(e^{(0)}_{{\theta}_s}\otimes t_1^{-rq}t_i)\Big)v_0=\sum_{q=1}^{l-(p+1)}(t_1^{r(-l+q)}t_2^{m_2}\cdots t_i^{m_i+1}\cdots t_n^{m_n}K_1)(e^{(0)}_{{\theta}_s}\otimes t_1^{-rq})v_0.
\end{aligned}$$

\noindent Proof of claim 2 follows by just applying $h^{(0)}_{\theta_s}\otimes t_i$ on both sides of assertion (1) and using assertion (1) again for $l>p+1$ when $j=0$.
Now, by applying Claim 2 in Equation (\ref{eq9}), assertion (2) follows.
\end{proof}
Let $\overline{\mathcal{Z}}$ and $\hat{\mathcal{Z}}$ are the subalgebras of $\mathcal{T}^+(\mu)$ generated by $\{t_1^{-rl} t^{\underline m}K_i: l>0,\,\underline m\in\Z^{n-1}_{\geq 0},\, i=2,3,\dots,n\}$ and $\{t_1^{-rl} t_iK_1: l>0,\, i=2,3,\dots,n\}$, respectively.

\begin{lemma}
We have $\overline{\mathcal{Z}}v_0=\hat{\mathcal{Z}}v_0$.
\end{lemma}
\begin{proof}
The proof follows with a similar technique as that of the proof of proposition 5.2 of \cite{SSS}.
\end{proof}
\begin{lemma}\label{lemma2}
We have $W_{\mathrm{loc}}^{\mu}(\Lambda_0,\underline{0})^+=\hat{\mathcal{Z}}U(\mathcal{T}_{\mathrm{aff}}(\mu)^-)v_0$.
\end{lemma}
\begin{proof}
    Note that $W_{\mathrm{loc}}^{\mu}(\Lambda_0,\underline{0})^+=U(\mathcal{T}_{\mathrm{aff}}(\mu)^-\otimes\C[t_2,\dots,t_n]).v_0$, where $\mathcal{T}_{\mathrm{aff}}(\mu)^-=\mathfrak{n}_0^-\oplus  \bigoplus _{{j}={0}}^{{r-1}}\mathfrak{g}_j\otimes t_1^{j-r}\C[t_1^{-r}])$ and it is easy to see that  $W_{\mathrm{loc}}^{\mu}(\Lambda_0,\underline{0})^+\supseteq\hat{\mathcal{Z}}U(\mathcal{T}_{\mathrm{aff}}(\mu)^-)v_0$, so we just need to prove that $W_{\mathrm{loc}}^{\mu}(\Lambda_0,\underline{0})^+\subseteq\hat{\mathcal{Z}}U(\mathcal{T}_{\mathrm{aff}}(\mu)^-)v_0$. To prove this, it is enough to show that $\mathcal{T}_{\mathrm{aff}}(\mu)^-\otimes\C[t_2,\dots,t_n].v_0\subseteq \hat{\mathcal{Z}}U(\mathcal{T}_{\mathrm{aff}}(\mu)^-)v_0.$ Let $x\otimes t_1^{-k}t^{\underline{m}}\in\mathcal{T}_{\mathrm{aff}}(\mu)^-\otimes\C[t_2,\dots,t_n]$, clearly $k\in\Z_{\geq0}$ and $\underline{m}\in\Z^{n-1}_{\geq0}$.
If $\underline{m}=\underline{0}$ then there is nothing to show and hence assume that $\underline{m}\neq\underline{0}.$
Let $k=rl-j$, where $l\geq0,\,\, 0\leq j\leq r-1$, we have $x\in\mathfrak{g}_j$. Since $\mathfrak{g}_j$ is an irreducible $\mathfrak{g}_0$- module so for $e_{\theta_s}^{(j)}\in \mathfrak{g}_j$, there exists $y\in U(\mathfrak{g}_0)$ such that $y.e_{\theta_s}^{(j)}=x$. Let $p=m_2+\cdots+m_n$. By using $y.v_0=0$ and assertion (1) of Lemma [\ref{prom 5}], we have
    \begin{equation}\label{eq13}
 x\otimes t_1^{-(rl-j)}t^{\underline m} v_0= \begin{cases}
    0, & \text{if $l\leq p$},\\[.5em]
r\displaystyle \sum_{q=1}^{l-p}(t_1^{r(-l+q)}t^{\underline m}K_1)(x\otimes t_1^{-(rq-j)})v_0, & \text{if $l> p$}.
  \end{cases}
\end{equation}
    Using Equation (\ref{eq13}), we get that $x\otimes t_1^{-k}t^{\underline{m}}.v_0\in\hat{\mathcal{Z}}U(\mathcal{T}_{\mathrm{aff}}(\mu)^-)v_0$ and hence  $\mathcal{T}_{\mathrm{aff}}(\mu)^-\otimes\C[t_2,\dots,t_n].v_0\subseteq \hat{\mathcal{Z}}U(\mathcal{T}_{\mathrm{aff}}(\mu)^-)v_0.$
\end{proof}
\begin{prop}\label{prop1}
    We have
    $$\mathrm{ch}_{q_1}W_{\mathrm{loc}}^{\mu}(\Lambda_0,\underline{a})\leq\mathrm{ch}_{q_1}W_{\mathrm{loc}}^{\mu}(\Lambda_0,\underline{a})^+ \leq \mathrm{ch}_{q_1}L(\Lambda_0)\left(\displaystyle\prod_{p>0}\frac{1}{1-q_1^p}\right)^{n-1},$$
     $$\mathrm{ch}_{q_1,q_2,\dots,q_n}W_{\mathrm{loc}}^{\mu}(\Lambda_0,\underline{a})\leq\mathrm{ch}_{q_1,q_2,\dots,q_n}W_{\mathrm{loc}}^{\mu}(\Lambda_0,\underline{a})^+ \leq \mathrm{ch}_{q_1}L(\Lambda_0)\left(\displaystyle\prod_{p>0,j=2}^{n}\frac{1}{1-q_1^pq_j}\right),$$
     where $\underline{a}\in{(\C^*)}^{n-1}$ and $L(\Lambda_0)$ is the level one integrable irreducible highest weight module of $\mathcal{T}_{\mathrm{aff}}(\mu)$ with highest weight $\Lambda_0$.
\end{prop}
\begin{proof}
The proof follows by using a similar argument as that of the proof of proposition 5.4 of \cite{SSS} and Lemma \ref{lemma2}.
\end{proof}
\section{Realization of Global Weyl Modules of Level One} \label{secv} 
In this section, we will obtain the realization of $W_{\mathrm{glob}}^{\mu}(\Lambda_0)$ in terms of representation of tensor product of certain vertex algebras. We begin with some definitions and results which are relevant to the result of this paper. We start with the following: 
\begin{definition}
A vertex algebra $V$ is a vector space with a distinguished vector $|0 \rangle$, a translation operator $T\in\mathrm{End}(V)$ and a  map $Y : V  \rightarrow \mathrm{End} \,\, V [[z,z^{-1}]]$, $Y(a, z) = \sum_{n \in \mathbb{Z}} a_{(n)} z^{-n -1}$ satisfying the following axioms:
\begin{enumerate}
\item $Y(a,z)b \in V[[z^{-1}]][z]$ for all $a, b \in V$.
\item $Y(|0 \rangle, z) $ is an identity map on $V$; $a_{(-1)} |0 \rangle = a$.
\item $T(a)=a_{(-2)}|0\rangle$, $[T,Y(a,z)]=\frac{\partial}{\partial z}Y(a,z)$ for all $a\in V$.
\item Borcherds identity:
$$\sum_{s= 0}^{\infty}{(-1)^j \binom {n} {s} \big( a_{(p +n -s)} (b_{(q+s)}c) - (-1)^{n} b_{(q+n-s)}(a_{(p+s)}c) \big)} = \sum_{s= 0}^{\infty}{\binom {p} {s}(a_{(q+p-s)}c)},
$$ for all $a,b, c \in V$ and $p,q,s \in \mathbb{Z}$.
\end{enumerate}
\end{definition}
The vector $|0 \rangle$ is called as a vacuum vector and the map 
$Y : V  \rightarrow \mathrm{End} \,\, V [[z,z^{-1}]]$ is called state-field correspondence. The following identities which can be derived from Borcherds identity and properties of translation operator play an important role in the theory of vertex algebras:
$$Y(a_{(-1)}b, z) =\, : Y(a,z) Y(b,z):,$$
$$Y(a_{(-s-1)}b, z) = \frac{1}{s!}:\Big( \frac{\partial^s}{\partial z^{s}} Y(a,z)\Big) Y(b, z):, a ,b \in V, s \geq 0 ,$$
where the expression 
$:Y(c,z) Y(d,z):$ for $c, d \in V$ is called the normal ordered product of fields $Y(c,z)$ and $Y(d, z)$ and is given by the following expression
$$ \sum_{n < 0}{c_{(n)}z^{-n-1} Y(d,z)} + \sum_{n \geq 0}{Y(c,z) d_{(n)} z^{-n-1}}.$$

Let $V_1$ and $V_2$ be vertex algebras, then the tensor product $V_1 \otimes V_2$ also has a vertex algebra structure, with 
$|0 \rangle \otimes |0 \rangle$ and 
$T \otimes I + I \otimes T$ play role of vacuum vector and translation operator, respectively. The state field correspondence is
given by 
$$Y(a \otimes b, z) = Y(a, z) \otimes Y(b, z) = \sum_{p, q \in \mathbb{Z}}{a_{(p)} \otimes b_{(q)} z^{-p- q -2}}.$$
In terms of modes we have the following expression
$$(a \otimes b)_{(m)} = \sum_{p \in \mathbb{Z}} a_{(p)} \otimes b_{(m- p -1)}.$$ 
\subsection{Heisenberg vertex algebras}
Let $\mathscr{H}$ be a finite dimensional vector space with a symmetric non-degenerate bi-linear form. The Lie algebra $\overline{\mathscr{H}}\coloneqq \mathscr{H} \otimes \mathbb{C}[t, t^{-1}] \oplus \mathbb{C} K$ with brackets
$$[h_1 \otimes t^m, h_2 \otimes t^{n}] = m(h_1 | h_2)\delta_{m, -n}K,\,\, [\overline{\mathscr{H}}, K] = 0,\,\, h_1, h_2 \in \mathscr{H}$$ is called a Heisenberg Lie algebra. The Fock space of $\overline{\mathscr{H}}$ is defined as $$\mathscr{F} \coloneqq \mathrm{Ind}_{ \mathscr{H}[t] \oplus \mathbb{C}K}^{\overline{\mathscr{H}}}\mathbb{C},$$ where $\mathscr{H}[t]$ acts trivially on 
$\mathbb{C}$ and $K$ acts as identity. The space $\mathscr{F}$ has the vertex algebra structure where the highest weight vector $\bf{1}$ plays the role of a vacuum vector. The state-field correspondence is given by
$$Y(h \otimes t^{-1}, z) = \sum_{k \in \Z}{h_{(k)} z^{-m-1}},\,\,\, \mathrm{where} \,\,\,  h \in \mathscr{H}, h_{(k)}= h \otimes t^{k}. $$
\subsection{Lattice vertex algebras}\label{subl}
 Let $\mathbb{M}$ be an integral lattice of rank $s$ with a non-degenerate symmetric bi-linear form $(\cdot| \cdot)$. Let us assume that $\mathbb{M}$ is an even lattice, i.e., $(h|h) \in 2 \mathbb{Z}$ for all $h \in \mathbb{M}$. Let 
$\mathscr{H} = \mathbb{C} \otimes_{\mathbb{Z}} \mathbb{M}$, and let $\overline{\mathscr{H}}$ be the Heisenberg Lie algebra. Consider the Fock space representation $\mathscr{F}$ of $\overline{\mathscr{H}}$. Let $\varepsilon: \mathbb{M} \times \mathbb{M} \rightarrow \{\pm 1\}$ be a bi-multiplicative cocycle with the property $\varepsilon(\alpha, \beta) \varepsilon(\beta, \alpha) = (-1)^{( \beta \mid \alpha)}.$ Let $\C_{\varepsilon} [\mathbb{M}]$ be the twisted group algebra
with multiplication $e^{\alpha}e^{\beta} = \varepsilon(\alpha, \beta) e^{\alpha + \beta}$. The lattice vertex algebra $V_{Q}$ defined as 
$$V_{\mathbb{M}} = \mathscr{F} \otimes_{\C}{\C_{\varepsilon} [\mathbb{M}]}.$$ 
The action of $\overline{\mathscr{H}}$ on $V_{\mathbb{M}}$ is given by $h_{(n)} (\eta \otimes e^{\beta}) = (h_{(n)}\eta+\delta_{n, 0} ( h |\beta )\eta) \otimes e^{\beta}$ for $h \in \overline{\mathscr{H}}$, $n \in\Z$, $\eta\in\mathscr{F}$ and $\beta\in \mathbb{M}$.
The state-field correspondence is given by:
$$Y( h_{(-1)} \otimes 1, z) = \sum_{p \in \mathbb{Z}} h_{(p)} z^{-p-1},$$
 and
$$Y (1 \otimes e^{\beta}, z) = e^{\beta} z^{\beta_0} \mathrm{exp}\big( -\sum_{n <0}{\beta_{(n)}\frac{z^{-n}}{n}} \big)\mathrm{exp}\big( -\sum_{n >0}{\beta_{(n)}\frac{z^{-n}}{n}} \big), $$
$h \in \overline{\mathscr{H}}$ and $\beta \in \mathbb{M}$. In the above expression, $e^{\beta}$ is the left multiplication operator by $1 \otimes e^\beta$ and $z^{\beta_0}$ acts as $z^{\beta_0}(\eta \otimes e^{\alpha}) = z^{(\beta | \alpha)} (\eta \otimes e^{\beta}),$
where $\eta \in \mathscr{F}$ and $\alpha\in \mathbb{M}$.

\subsection {Twisted representations of vertex algebras}
An invertible linear map $S$ from a vertex algebra $V$ to itself is called an automorphism if $S (|0\rangle) = |0\rangle$ and $S(a_{(n)} b) = S(a)_{(n)} S(b)$.
Suppose that $S$ is of finite order, i.e., $S^{N} = I$ for some $N \in \mathbb{N}$. A vector space $M$ is called a $S$-twisted  representation  of $V$ if there exists a linear map
$Y_{M}: V \rightarrow \mathrm{(End}\,\, M)[[z^{\frac{1}{N}}, z^{\frac{-1}{N}}]] ,$
$$Y_{M} (a, z) = \sum_{s \in \frac{1}{N} \mathbb{Z}}{a_{(s)}^{M} z^{-s-1}}, a \in V,$$
which satisfies the following axioms:
\begin{enumerate}
\item $a^{M}_{(s)}.v = 0$ for $s >>0$ for all $a \in V$ and for all $v \in M$.
\item $Y_{M} (|0\rangle, z)$ is identity on $M$ and $Y_{M} (Sa, z) = Y_{M} (a, e^{2 \pi i}z)$.
\item Twisted Borcherds identity:
$$\sum_{s= 0}^{\infty}{(-1)^s \binom {n} {s} \left( a^{M}_{(p +n -s)} (b^{M}_{(q+s)}c )- (-1)^{n+s} b^{M}_{(q+n-s)}(a^{M}_{(p+s)}c) \right)} = \sum_{s= 0}^{\infty}{\binom {p} {s}(a_{(n+s)}b)^{M}_{(p+q-s)} c},$$
for all $a, b \in V, c \in M, n \in \mathbb{Z}, p , q  \in \frac{1}{N} \mathbb{Z}$, provided that $S (a) = e^{-2 \pi p z }a$.
 \end{enumerate}
\subsection{}
 
\noindent
 Now let $\mathfrak{g}$ be a finite dimensional simple Lie algebra of type $A, D$ or $E$. Let $R,\,\, Q= \mathbb{Z}R,\,\, \mathfrak{h} = \mathbb{C}\otimes_{\mathbb{Z}}Q$ denote set of roots, root lattice and Cartan subalgebra of $\mathfrak{g}$, respectively. Let $(\cdot, \cdot)$ be the normalised Killing form of
 $\mathfrak{g}$. So we have $(\alpha | \alpha) =2$ for all $\alpha \in R$.
 Let $\overline{\mathfrak{h}} = \mathfrak{h}\otimes \mathbb{C}[t, t^{-1}] \oplus 
 \mathbb{C} K$ be Heisenberg algebra and $\mathscr{F}$ be its Fock space representation. Let $V_{Q}\coloneqq \mathscr{F} \otimes_{\mathbb{C}} \mathbb{C}_{\varepsilon}[Q]$ be the lattice vertex algebra associated with $Q$. 
 We have the following important result from \cite{BKIRK} (Corollary 3.3) :
\begin{prop} \label{impr}
Let $M$ be a any $\sigma$-twisted module of $V_{Q}$. Then the span of modes of the fields $Y_M(h, z)$, $h \in \mathfrak{h}$ and $Y_M(e^{\alpha}, z)$, $\alpha \in R$ form a representation of $\hat{\mathfrak{g}}'(\mu^{-1})$.
\end{prop}
 
 Recall that $\mathrm{span}_\C\{h\otimes t^{-1}\otimes 1, 1\otimes e^{\alpha}: h\in\mathfrak{h},\alpha\in R \}\subseteq V_Q$ is isomorphic to the Lie algebra $\mathfrak{g}$ with the assignments $h\otimes t^{-1}\otimes 1\mapsto h$ and $1\otimes e^{\alpha}\mapsto e_\alpha\in \mathfrak{g}_{\alpha}$ for all $h\in\mathfrak{h}$ and $\alpha\in R$. For any $a\in\mathfrak{g}$, $Y(a,z)$ will always mean the field of an element of $V_Q$ that corresponds to $a$ under the above assignments.
 Let $\mu$ be the Dynkin diagram automorphism of $\mathfrak{h}$ of order $r$. 
 Using the results in \cite{BBVK, TKVJ} and Proposition \ref{impr} we have the following fact:
 \begin{itemize}
\item  There exists a $\mu$-twisted module $M$ of $V_{Q}$ which considered as $\hat{\mathfrak{g}}(\mu^{-1})$-module is isomorphic to the basic representation $L(\Lambda_0)$. The action is given by the map $\sigma: \hat{\mathfrak{g}}(\mu^{-1}) \rightarrow \mathrm{End} (M)$ with the assignments
$$a^{(j)} \otimes t^{mr + j} \mapsto  a^{(j)}_{(\frac{{mr + j}}{r})}, a^{(j)} \in \mathfrak{g}_j$$
$$K \mapsto \frac{1}{r}\mathrm{Id} ,$$
$$d \mapsto - r L_0,$$
where for dual bases $u_i$ and $u^{i}$, $i\in I_0$ of 
$\mathfrak{h}_{0}$, the operator $L_0$ is defined as follows
$$L_0 \coloneqq \sum_{i = 1}^{l}{\Big( \frac{1}{2}{u_{i}}_{(0)}u^{i}_{(0)} + \sum_{n \in \mathbb{Z}_{>0}}{u_{i}}_{(-n)}u^{i}_{(n)}}\Big).$$
 
\end{itemize}
 \subsection{Lattice vertex algebras in  toroidal setup}
 Let $\Gamma$ be an integral lattice of rank $2n-2$ spanned by $\delta_2, \ldots, \delta_n, \Lambda_2, \ldots, \Lambda_n$ with a symmetric nondegenerate bilinear form 
 $\langle ,\,\, \rangle$ with values: $$\langle \delta_i, \delta_j \rangle = 0 = \langle \Lambda_i, \Lambda_j \rangle,\,\, \langle \delta_1, \Lambda_j \rangle = \delta_{i, j} .$$ Let $\mathfrak{p} = \mathbb{C} \otimes_{\mathbb{Z}} \Gamma$ and
 $\overline{\mathfrak{p}} = \mathfrak{p} \otimes \mathbb{C}[t, t^{-1}] \oplus \mathbb{C} c$ be the Heisenberg Lie algebra and let $S(t^{-1}\mathfrak{p}[t^{-1}])$ be the Fock space representation of $\overline{\mathfrak{p}}$. For $h \in \mathfrak{p}$ the action of $h \otimes t^{-s}, s>0$ is given by the multiplication by $h \otimes t^{-s}$, $c$ acts as identity, and  $h \otimes t^0$ acts trivially,
 $h \otimes t^{s}, s>0$ acts as a derivation given by $\beta\otimes t^{-m} \mapsto s\delta_{s, k} \langle h, \beta\rangle, m >0$.
 Let $\varepsilon : \Gamma \times \Gamma
 \rightarrow  {\pm 1}$ be a bimultiplicative function with values 
 $$\varepsilon (\delta_i, \Lambda_i) = -1$$ and values on all other generators equal to $1$.
 
 Let $V_{\Gamma}\coloneqq S(t^{-1}\mathfrak{p}[t^{-1}]) \otimes \mathbb{C}_{\varepsilon}[\Gamma]$. The action of $\overline{\mathfrak{p}}$ extends to $V_{\Gamma}$ as follows:
 for $h \in \mathfrak{p}, \beta \in \Gamma, s\in \mathbb{Z}, s \neq 0,$
 $$h\otimes t^{s}( u \otimes e^{\beta}) = (h \otimes t^{s}.u) \otimes e^{\beta};$$ 
 $$h \otimes t^{0}(u \otimes e^{\beta}) = \langle h, \beta \rangle u \otimes e^{\beta}. $$
 Let for ${\bf{q}}= (q_2, \ldots, q_n)$, ${\bf{q}} \delta \coloneqq q_2 \delta_2 + \cdots +q_n  \delta_n$ and ${\bf{q}} \Lambda \coloneqq q_2 \Lambda_2 + \cdots +q_n  \Lambda_n$.
 Let 
 $$Y(u \otimes t^{-1} \otimes 1, z) = \sum_{k \in \mathbb{Z}}{(u \otimes t^n)z^{-n-1}}$$
 $$Y (1 \otimes e^{{\bf{q}} \delta}, z) = e^{{\bf{q}} \delta} z^{({\bf{q}} \delta)_{ (0)}} \mathrm{exp}\big( -\sum_{n <0}{\frac{z^{-n}}{n} ({\bf{q}} \delta)_{(n)}} \big)\mathrm{exp}\big( -\sum_{n >0}{\frac{z^{-n}}{n} ({\bf{q}} \delta)_{(n)}} \big) 
 = \sum_{k \in \mathbb{Z}}{e^{{\bf{q}} \delta}_{(k)} z^{-k -1}}$$
 where $e^{{\bf{q}} \delta}$ and $z^{({\bf{q}} \delta)_{(0)}}$
defined similarly as in Subsection \ref{subl}. Similarly, $Y (1 \otimes e^{{\bf{q}} \Lambda}, z)$ is defined. The following expressions give the n-th product identities in $V_{\Gamma}$:
$$\begin{aligned} {e^{\bf{q}\delta}}_{(-1)}e^{\bf{p} \delta} &=e^{(\bf{q + p)} \delta},\,\,\,\,
{e^{\bf{q}\delta}}_{(-2)}e^{\bf{p} \delta} = ({{\bf{q}} \delta})_{(-1)}e^{(\bf{q + p}) \delta},\\ (\alpha_{(-1)}e^{\bf{q}\delta})_{(n)}(\beta_{(-1)}e^{\bf{p}\delta}) &= 0
\text{ for all } n \in \mathbb{Z}_{\geq 0}, \,\,\,\,\alpha, \beta \in \{1, \delta_2, \ldots, \delta_n\},\\
({\Lambda_i}_{(-1)} e^{\bf{q}\delta})_{(0)}e^{\bf{p} \delta} &= q_i e^{\bf{(q + p)} \delta},\,\,\,\, ({\Lambda_i}_{(-1)} e^{\bf{q}\delta})_{(n)}e^{\bf{p} \delta} = 0 \text{ for all } n \geq 1,\\ ({\Lambda_i}_{(-1)} e^{\bf{q}\delta})_{(0)}({\delta_j}_{(-1)}e^{\bf{p} \delta}) &=  p_i {\delta_j}_{(-1)}e^{(\bf{q + p}) \delta} + \delta_{i, j} ({\bf{q} \delta})_{(-1)}e^{\bf{(q + p)} \delta},\\
({\Lambda_i}_{(-1)} e^{\bf{q}\delta})_{(1)}({\delta_j}_{(-1)}e^{\bf{p} \delta}) &= \delta_{i, j}e^{\bf{(q + p)} \delta},\,\,\,\, ({\Lambda_i}_{(-1)} e^{\bf{q}\delta})_{(n)}({\delta_j}_{(-1)}e^{\bf{p} \delta}) = 0, \text{ for all }n \geq 2,\\
({\Lambda_i}_{(-1)} e^{\bf{q}\delta})_{(0)}({\Lambda_j}_{(-1)}e^{\bf{p} \delta})
&= (-q_j {\Lambda_{i}}_{(-1)} + p_i {\Lambda_{j}}_{(-1)} - p_i q_j {({\bf{q}}\delta})_{(-1)})e^{({\bf{q+p}})\delta},\\
({\Lambda_i}_{(-1)} e^{\bf{q}\delta})_{(1)}({\Lambda_j}_{(-1)}e^{\bf{p} \delta}) &= 
-p_iq_j e^{({\bf{q+p}})\delta}, \,\,\,\, ({\Lambda_i}_{(-1)} e^{\bf{q}\delta})_{(n)}({\Lambda_j}_{(-1)}e^{\bf{p} \delta}) =0  \text{ for all } n \geq 2.
\end{aligned}$$
\subsection{\texorpdfstring{Representation of $\mathcal{T}(\mu)$}{}}
 Let $\Gamma_1$ be a integral lattice of rank $n-1$ spanned by $\delta_2, \ldots, \delta_n$. Let $H_1 = \mathbb{C}
 \otimes_{\mathbb{Z}} \Gamma_1$ be the complex vector space, and let $\overline{H}_1 = H_1\otimes \C[t,t^{-1}] \oplus \mathbb{C} c$ be the Heisenberg Lie algebra associated with $H_1$. Let 
 $\mathscr{F}_1$ be the Fock space representation of $\overline{H}_1$. Let $V_{\Gamma_1} = \mathscr{F}_1 \otimes \mathbb{C}[\Gamma_1]$ be the lattce vertex algebra associated with $\Gamma_1$. 
 Let $\mathcal{M} \coloneqq \mathbb{C}[\delta_{2} (k), \ldots, \delta_{n}(k): k \in \mathbb{Z}_{< 0}]$ be a polynomial algebra freely generated with variables $\{\delta_i (k): 2 \leq i \leq n, k \in \mathbb{Z}_{< 0} \}$. Consider the Laurent polynomial ring
 in $n-1$ variables $\mathbb{C}[\tau_{2}^{\pm 1}, \ldots, \tau_{n}^{\pm 1}]$. Then it is easy to see that $V_{\Gamma_1} \cong \mathcal{M} \otimes \mathbb{C}[\tau_{2}^{\pm 1}, \ldots, \tau_{n}^{\pm 1}]$ as vector spaces. Let $d^{\mathcal{M}}$ denote the operator on $\mathcal{M}$ which measures degree on $\mathcal{M}$. For $2 \leq i \leq n,$ let $d^{(i)}$ denote the operator on 
 $\mathbb{C}[\tau_{2}^{\pm 1}, \ldots, \tau_{n}^{\pm 1}]$ which counts the degree of variable ${\tau_{i}}$.
 
 We have the following:
 \begin{thm}
  The space $\mathcal{G}\coloneqq M \otimes V_{\Gamma_1} = M \otimes \mathcal{M} \otimes \mathbb{C}[\tau_{2}^{\pm 1}, \ldots, \tau_{n}^{\pm 1}]$ admits a structure of $\mathcal{T}(\mu^{-1})$ representation via the map $\Psi: \mathcal{T}(\mu^{-1}) \rightarrow \mathrm{End}\,\, \mathcal{G}$ with the following assignments:
  for $a^{(i)} \in \mathfrak{g}_i,$
$$\begin{aligned}
\sum_{k \in \Z}a^{(i)} \otimes t_1^{rk +i} t^{\underline{m}} z^{-(rk +i) -1} &\mapsto \sum_{k \in \mathbb{Z}}{a^{(i)}_{(\frac{rk+i}{r})} z^{-(\frac{rk+i}{r}) -1}} \otimes Y(e^{\underline{m} \delta}, z) ;\\
    \sum_{k \in \mathbb{Z}}{ t_1^{rk} t^{\underline{m}}K_1 z^{-k}}
  &\mapsto  \frac{1}{r}\mathrm{Id} \otimes Y(e^{\underline{m}\delta}, z);\\
  \sum_{k \in \mathbb{Z}}{ t_1^{rk} t^{\underline{m}}K_i z^{-k}}
  &\mapsto  \mathrm{Id} \otimes  Y({\delta_{i}}_{(-1)}e^{\underline{m}\delta}, z) ,\, 2 \leq i \leq n;\\
d_1&\mapsto  r(-L_0 \otimes \mathrm{Id}  +  \mathrm{Id} \otimes  d^{\mathcal{M}}) ;\\
d_i &\mapsto  \mathrm{Id}\otimes d^{(i)},\, 2 \leq i \leq n .
  \end{aligned}$$
  \end{thm}
 \begin{proof}
 The above assignments can be rewritten as follows:
 for $a^{(i)} \in \mathfrak{g}_i$,
 $$a^{(i)} \otimes t_1^{rs+i} t^{\underline{m}} \mapsto  (a^{(i)} \otimes e^{\underline{m}\delta} )_{\big(\frac{rs+i}{r} \big)},$$
 where $$(c \otimes d)_{(\frac{k}{l})} \coloneqq \sum_{p \in \mathbb{Z}}{c_{(\frac{k}{l} -p -1)} \otimes d_{(p)}},$$
 $$t_1^{rs} t^{\underline{m}}K_1  \mapsto 
 \frac{1}{r}|0\rangle_{(-1)} \otimes e^{\underline{m} \delta}_{(s-1)},$$
$$t_1^{rs} t^{\underline{m}}K_i  \mapsto  
 |0\rangle_{(-1)} \otimes ({\delta_{i}}_{(-1)}e^{\underline{m} \delta})_{(s)} ,\,\,\, 2 \leq i \leq n.$$
 Now we need to show that $\Psi$ preserves the brackets in $\mathcal{T}(\mu^{-1})$. More precisely:
 $$\begin{aligned}
\Psi[x^{(j)} \otimes t_1^{rs + j}t^{\underline{m}}, y^{(k)} \otimes t_1^{rl+k}t^{\underline{m}}] &= [\Psi (x^{(j)} \otimes t_1^{rs + j}t^{\underline{m}}), \Psi(y^{(k)} \otimes t_1^{rl+k}t^{\underline{m}})],\\
 \Psi[t_1^{rs} t^{\underline{m}}K_i, x^{(j)} \otimes t_1^{rl + j}t^{\underline{m}} ] &= [\Psi(t_1^{rs} t^{\underline{m}}K_i), \Psi(x^{(j)} \otimes t_1^{rl + j}t^{\underline{m}})], \\
 \Psi[d_i, x^{(j)} \otimes t_1^{rs + j}t^{\underline{m}}] &= 
 [\Psi(d_i), \Psi(x^{(j)} \otimes t_1^{rs + j}t^{\underline{m}})],\\
\Psi[d_i,  t_1^{rs}t^{\underline{m}}K_l] &= [\Psi(d_i),  \Psi(t_1^{rs}t^{\underline{m}}K_l)],\end{aligned}$$
where $0 \leq j,k \leq r-1,\,1 \leq i,l \leq n.$
 But we will be done if we show that 
$$ \begin{aligned}\Psi[d_1, x^{(j)} \otimes t_1^{rs + j}t^{\underline{m}}] &= 
 [\Psi(d_1), \Psi(x^{(j)} \otimes t_1^{rs + j}t^{\underline{m}})],\\
 \Psi[d_1,  t_1^{rs}t^{\underline{m}}K_i] &= [\Psi(d_1),  \Psi(t_1^{rs}t^{\underline{m}}K_i)], 1 \leq i \leq n,\end{aligned}$$ as the proof for all other identities follows from Theorem 4.2 of \cite{BKIRK}.
 Hence we obtain  $$\begin{aligned} [\Psi(d_1),& \Psi(x^{(j)} \otimes t_1^{rs + j}t^{\underline{m}})] = [r(-L_0 \otimes \mathrm{Id}  +  \mathrm{Id} \otimes  d^{\mathcal{M}}), (x^{(j)} \otimes e^{\underline{m} \delta})_{\frac{rs+j}{r}}]\\
 &= \Big[r(-L_0 \otimes \mathrm{Id}  +  \mathrm{Id} \otimes  d^{\mathcal{M}}), \sum_{p \in \mathbb{Z}}{x^{(j)}_{(\frac{rs+j}{r} - p -1)}\otimes e^{\underline{m}\delta}_{(p)}} \Big]\\
 &= \sum_{p \in \mathbb{Z}}{r\Big(\frac{rs+j -rp -r}{r}\Big) x^{(j)}_{\frac{rs+j}{r} - p - 1}}\otimes e^{\underline{m}\delta}_{(p)} + 
 \sum_{p \in \mathbb{Z}}{x^{(j)}_{(\frac{rs+j}{r} - p -1)}}\otimes r(p+1)e^{\underline{m}\delta}_{(p)}\\
 &= (rs+j)\sum_{p \in \mathbb{Z}}{x^{(j)}_{(\frac{rs+j}{r} - p -1)}\otimes e^{\underline{m}\delta}_{(p)}}\\
 &= (rs+ j)(x^{(j)} \otimes e^{\underline{m} \delta})_{(\frac{rs+j}{r})}\\
 &=\Psi[d_1, x^{(j)} \otimes t_1^{rs + j}t^{\underline{m}}].
 \end{aligned}$$
 Similarly,
 $$
 \begin{aligned}\relax [\Psi(d_1),  \Psi(t_1^{rs}t^{\underline{m}}K_1)] &= 
 \Big[r(-L_0 \otimes \mathrm{Id}  +  \mathrm{Id} \otimes  d^{\mathcal{M}}),\frac{1}{r}|0\rangle_{(-1)} \otimes e^{\underline{m} \delta}_{(s-1)} \Big] \\
 &= \frac{1}{r}|0\rangle_{(-1)} \otimes rs \,(e^{\underline{m} \delta}_{(s-1)}) = s(|0\rangle_{(-1)} \otimes e^{\underline{m} \delta}_{(s-1)})\\
 &= rs(\frac{1}{r}|0\rangle_{(-1)} \otimes e^{\underline{m} \delta}_{(s-1)}) = \Psi[d_1,  t_1^{rs}t^{\underline{m}}K_1].
 \end{aligned}$$
 Lastly, for $2 \leq i \leq n $, we obtain
$$ [\Psi(d_1),  \Psi(t_1^{rs}t^{\underline{m}}K_i)]= 
\Big[r(-L_0 \otimes \mathrm{Id}  +  \mathrm{Id} \otimes  d^{\mathcal{M}}), |0\rangle_{(-1)} \otimes ({\delta_{i}}_{(-1)}e^{\underline{m} \delta})_{(s)}\Big]$$
$$= |0\rangle_{(-1)} \otimes rs ({\delta_{i}}_{(-1)}e^{\underline{m} \delta})_{(s)} = rs (|0\rangle_{(-1)} \otimes ({\delta_{i}}_{(-1)}e^{\underline{m} \delta})_{(s)}) = \Psi[d_1,  t_1^{rs}t^{\underline{m}}K_i].$$
 \end{proof}
 Let ${\bf{v}}$ denote the highest weight vector of $L(\Lambda_0)$, the basic representation of $\mathfrak{g}(\mu^{-1})$.
 
  \begin{thm}
  The map 
  $ \Phi: W_{\mathrm{glob}}^{\mu^{-1}}(\Lambda_0) \rightarrow \mathcal{G}$ defined by $\Phi(v_{\Lambda_0}) = {\bf{v}} \otimes 1 \otimes 1$ is 
  a surjective $\mathcal{T}(\mu^{-1})$-module map.
  \end{thm}
  \begin{proof}
  As the space $\mathcal{G}$ is a cyclic $\mathcal{T}(\mu^{-1})$-module with the cyclic vector ${\bf{v}} \otimes 1 \otimes 1$, it is enough to prove that the map $\Phi$ satisfies the defining relations of  $W_{\mathrm{glob}}^{\mu^{-1}}(\Lambda_0)$.
  For $i\in I_0,\, \underline{m} \in \mathbb{Z}^{n-1},$ consider
  $$\begin{aligned}
      e_{i}\otimes t^{\underline{m}}.({\bf{v}} \otimes 1 \otimes 1) &= (e_i \otimes e^{\underline{m}\delta})_{(0)}  ({\bf{v}} \otimes 1 \otimes 1)\\
  &= (\sum_{k \in \mathbb{Z}}{{e_i}_{(-k-1)}} \otimes e^{\underline{m}\delta}_{(k)})({\bf{v}} \otimes 1 \otimes 1) = 0  \end{aligned}$$
  as $e^{\underline{m}\delta}_{(k)}.1 = 0$ for all $k \in \mathbb{Z}_{\geq 0}$ and ${e_i}_{(k)}.{\bf{v}} = 0$ for all $k \in \mathbb{Z}_{\geq 0}$.
  By similar reasoning we have $e_0.({\bf{v}} \otimes 1 \otimes 1)=h_i.({\bf{v}} \otimes 1 \otimes 1)=f_0^{2}.({\bf{v}} \otimes 1 \otimes 1) = f_i.({\bf{v}} \otimes 1 \otimes 1) = 0$ for $i \in I_{0}$. Consider
  $$\begin{aligned}
      h_0.({\bf{v}} \otimes 1 \otimes 1)&=(-h^{(0)}_{\theta_s}+rK_1)({\bf{v}} \otimes 1 \otimes 1)\\
      &=|0\rangle_{(-1)}\otimes (1 \otimes 1)_{(-1)}({\bf{v}} \otimes 1 \otimes 1)={\bf{v}} \otimes 1 \otimes 1
      \end{aligned}$$ as $h^{(0)}_{\theta_s}.({\bf{v}} \otimes 1 \otimes 1)=0$. 
      
      For $2 \leq i \leq n,\,\, \underline{m} \in \mathbb{Z}^{n-1},$ consider
  $$\begin{aligned}
t^{\underline{m}}K_i.({\bf{v}}\otimes1\otimes1)&=|0\rangle_{(-1)}\otimes({\delta_i}_{(-1)}e^{\underline{m}}\delta)_{(-1)}({\bf{v}} \otimes 1 \otimes 1)\\
\mathrm{but} \,\, Y({\delta_i}_{(-1)}e^{\underline{m}}\delta,z) & =\, :Y(\delta_i,z)Y(e^{\underline{m}\delta},z):\\
 &= \sum_{k<0}{\delta_i}_{(k)}z^{-k-1}Y(e^{\underline{m}\delta},z)+ \sum_{k\geq0}Y(e^{\underline{m}\delta},z){\delta_i}_{(k)}z^{-k-1}.
\end{aligned}$$
Now as  ${\delta_i}_{(k)}.1 \otimes 1 =0$ for all $k \in \mathbb{Z}_{\geq 0}$ and $e^{\underline{m}\delta}_{(k)}.1 \otimes 1 = 0$ for all $k \in \mathbb{Z}_{\geq 0}$, we get that 
$$({\delta_i}_{(-1)}e^{\underline{m}}\delta)_{(-1)} (1 \otimes 1) = 0.$$

Lastly, for $2\leq i\leq n$, consider
$$d_i.({\bf{v}} \otimes 1 \otimes 1)=\mathrm{Id}\otimes d^{(i)}({\bf{v}} \otimes 1 \otimes 1)=0.$$
 \end{proof}
\subsection{\texorpdfstring{$\mathcal{G}$ as $A^{\mu^{-1}}(\Lambda_0)$-module}{}}
We have $$
\begin{aligned}
h_{0, \underline{m}}({\bf{v}} \otimes 1 \otimes 1) &= -h^{(0)}_{\theta_s} \otimes t^{\underline{m}} + r K_1 t^{\underline{m}}({\bf{v}} \otimes 1 \otimes 1)\\
&= r \frac{1}{r}(|0 \rangle_{(-1)} \otimes e^{\underline{m}\delta}_{(-1)})({\bf{v}} \otimes 1 \otimes 1) = 
{\bf{v}} \otimes 1 \otimes \tau^{\underline{m}}.
\end{aligned}$$
For $\underline{a} \in ({\mathbb{C}^{*})}^{n-1}$, define
$\mathcal{G}_{\underline{a}}\coloneqq \mathcal{G} \otimes_{A^{\mu^{-1}}(\Lambda_0)}\mathbb{C}_{\underline{a}}$. The following is an immediate
\begin{prop}
Let $\underline{a} \in ({\mathbb{C}^{*})}^{n-1}$. The $q_1$ character of 
$\mathcal{G}_{\underline{a}}$ is given by the formula
$$\mathrm{ch}_{q_1} (\mathcal{G}_{\underline{a}}) = \mathrm{ch}_{q_1} (L(\Lambda_0)) \left( \prod_{s>0}{\frac{1}{1-q_1^{s}}}\right)^{n-1},$$
where $L(\Lambda_0)$ is the level one irreducible integrable highest weight module of $\hat{\mathfrak{g}}(\mu^{-1})$.
\end{prop}
We have the following:
\begin{cor} \label{1cr}
 As a $\mathcal{T}(\mu^{-1})$-module 
 $\mathcal{G}_{\underline{a}} \cong W^{\mu^{-1}}_{\mathrm{loc}}{(\Lambda_0, \underline{a})}$, where $\underline{a} \in {(\mathbb{C}^*)}^{n-1}$. Further, we have
 $$
 \begin{aligned}
\mathrm{ch}_{q_1} (\mathcal{G}_{\underline{a}}) &= \mathrm{ch}_{q_1}
 W^{\mu}_{\mathrm{loc}}{(\Lambda_0, \underline{a})^+} = \mathrm{ch}_{q_1}
 W^{\mu}_{\mathrm{loc}}{(\Lambda_0, \underline{a})},\\
\mathrm{ch}_{q_1, q_2, \ldots, q_n} (\mathcal{G}_{\underline{a}}) &= \mathrm{ch}_{q_1, q_2, \ldots, q_n}
 W^{\mu}_{\mathrm{loc}}{(\Lambda_0, \underline{a})^+} = \mathrm{ch}_{q_1, q_2, \ldots, q_n}
 W^{\mu}_{\mathrm{loc}}{(\Lambda_0, \underline{a})} \\
  &= \mathrm{ch}_{q_1} (L(\Lambda_0)) \left( \prod_{s>0, i = 2}^{n}{\frac{1}{1-q_1^{s} q_i}}\right).
 \end{aligned}$$
 
 \end{cor}
\begin{proof}
The proof is exactly similar to that of the proof of Theorem 6.8 of \cite{SSS}.
\end{proof}
\begin{thm} \label{gthm}
As a $\mathcal{T}(\mu^{-1})$-module $W^{\mu^{-1}}_{\mathrm{glob}}(\Lambda_{0}) \cong \mathcal{G}$.
\end{thm}
\begin{proof}
The argument is similar to that of the proof of Theorem 4.8 of \cite{KR} which uses Corollary \ref{1cr} and the Nakayama lemma.
\end{proof}
The following theorem readily follows from Corollary \ref{1cr} and Theorem \ref{gthm}. 
\begin{thm}
Let $\mathcal{T}(\mu)$ be twisted toroidal Lie algebra whose underlying finite dimensional simple Lie algebra is of type $A_{2 \ell +1}$, or $D_{\ell}$. Let $\mathcal{T}_{\mathrm{aff}}(\mu)$ be its twisted affine Kac-Moody subalgebra of type $A_{2\ell+1}^{(2)}$, or
$D_{\ell}^{(2)}$ or $D_{4}^{(3)}$. Let $L(\Lambda_0)$ be the basic representation of $\mathcal{T}_{\mathrm{aff}}(\mu)$. Then we have 
$$\mathrm{ch}_{q_1}  W^{\mu}_{\mathrm{loc}}(\Lambda_0, \underline{a}) = 
\mathrm{ch}_{q_1} L(\Lambda_0)\left( \prod_{s>0}{\frac{1}{1-q_1^{s}}}\right)^{n-1},$$
$$\mathrm{ch}_{q_1, q_2, \ldots, q_n}
 W^{\mu}_{\mathrm{loc}}{(\Lambda_0, \underline{a})} = 
 \mathrm{ch}_{q_1} L(\Lambda_0) \left( \prod_{s>0, i = 2}^{n}{\frac{1}{1-q_1^{s} q_i}}\right).$$
\end{thm}
With the assumption of the hypothesis of above theorem we have the following:
\begin{cor}
 
$$\mathrm{ch}_{q_1}  W^{\mu}_{\mathrm{loc}}(\Lambda_0, \underline{a}) = 
e^{\Lambda_0} \prod_{p = 1}^{\infty}{\left(\frac{1}{1- q_1^{p}} \right) ^{\mathrm{mult} \, p \delta_1}} \left( \prod_{s>0}{\frac{1}{1-q_1^{s}}}\right)^{n-1},$$
$$\mathrm{ch}_{q_1, q_2, \ldots, q_n}
 W^{\mu}_{\mathrm{loc}}{(\Lambda_0, \underline{a})} = 
 e^{\Lambda_0} \prod_{p = 1}^{\infty}{\left(\frac{1}{1- q_1^{p}} \right) ^{\mathrm{mult} \, p \delta_1}} \left( \prod_{s>0, i = 2}^{n}{\frac{1}{1-q_1^{s} q_i}}\right).$$
\end{cor}
\begin{proof} Immediate from Proposition 12.13 of \cite{KAC}.
\end{proof}

\bibliographystyle{plain}
\bibliography{reference}

\bigskip
\noindent
riteshp@iitk.ac.in,  sachinsh@iitk.ac.in\\
Department of mathematics and statistics,
IIT kanpur.

\end{document}